\documentclass[11pt]{amsart}

\usepackage{amsmath, amsthm, amssymb}
\usepackage[alphabetic]{amsrefs}
\usepackage{hyperref}
\usepackage[shortlabels]{enumitem}
\usepackage[margin=3cm]{geometry}

\newtheorem{theorem}{Theorem}[section]
\newtheorem{lemma}[theorem]{Lemma}
\newtheorem{corollary}[theorem]{Corollary}
\newtheorem{proposition}[theorem]{Proposition}

\theoremstyle{remark}
\newtheorem{remark}[theorem]{Remark}

\newtheorem{example}[theorem]{Example}

\numberwithin{equation}{section}

\newcommand{\N}{\mathbb{N}}

\newcommand{\R}{\mathbb{R}}
\newcommand{\C}{\mathbb{C}}

\newcommand{\maj}{\prec_{\mathrm u}}
\newcommand{\smaj}{\prec_{\mathrm c}}
\newcommand{\smajtau}{\prec_{\mathrm T}}

\newcommand{\co}{\mathrm{co}}
\newcommand{\T}{\mathrm{T}}
\newcommand{\U}{\mathrm{U}}
\newcommand{\sa}{\mathrm{sa}}
\renewcommand{\epsilon}{\varepsilon}
\renewcommand{\leq}{\leqslant}
\renewcommand{\geq}{\geqslant}

\title{Majorization in C*-algebras}

\author{Ping Wong Ng}
\address{\hskip-\parindent
Department of Mathematics,
University of Louisiana at Lafayette,
Lafayette, USA.}
\email{png@louisiana.edu}

\author{Leonel Robert}
\address{\hskip-\parindent
Department of Mathematics,
University of Louisiana at Lafayette,
Lafayette, USA.}
\email{lrobert@louisiana.edu}

\author{Paul Skoufranis}
\address{\hskip-\parindent
Department of Mathematics and Statistics,
York University, Toronto, Canada.}
\email{pskoufra@yorku.ca}

\begin{document}
\begin{abstract}
We investigate the closed convex hull of unitary orbits of selfadjoint elements in arbitrary unital C*-algebras.  Using a notion of majorization against unbounded traces, a characterization of these closed convex hulls is obtained.  
Furthermore, for C*-algebras satisfying Blackadar's strict comparison of positive elements by traces or for collections of C*-algebras with a uniform bound on their nuclear dimension, an upper bound  for the number of unitary conjugates  in a convex combination required to approximate an element in the closed convex hull within a given error is shown to exist. This property, however,
 fails for certain ``badly behaved"  simple nuclear C*-algebras. 
\end{abstract}
\maketitle

\section{Introduction}
The relation of majorization between selfadjoint matrices is an important and well studied relation (see \cite{ando} and references therein).  It is thus natural to pursue its study in the more general realm of operator algebras. This has been done for von Neumann algebra factors (\cite{kamei, hiai-nakamura}) and for various classes of simple C*-algebras (\cite{skoufranis,ng-skoufranis}).
 A basic result on matrix majorization due to Uhlmann gives two equivalent ways of defining the majorization relation: Given selfadjoint matrices $a$ and $b$,  the following conditions on $a$ and $b$ are equivalent:
 \begin{enumerate}[(1)]
 \item
 	  $a$ belongs to the convex hull of
the unitary conjugates of $b$,
\item
 $\mathrm{Tr}(a)=\mathrm{Tr}(b)$ and 
$\mathrm{Tr}((a-t)_+)\leq \mathrm{Tr}((b-t)_+)$
for all $t\in \R$. Here $(a-t)_+$ is the element obtained from $a$ by functional calculus with the function $x\mapsto (x-t)_+:=\max(x-t,0)$ and $\mathrm{Tr}$ is the trace. 
\end{enumerate}
When either of these  conditions holds  $a$ is said to be majorized by $b$.  
We show in this paper  that  the equivalence above   has a natural generalization to arbitrary C*-algebras.  In order to formulate a suitable version of (2)  we must now look at possibly unbounded traces. Let $A$ be a C*-algebra. We call  a map $\tau\colon A_+\to [0,\infty]$ a trace if it is linear (additive, $\R_+$-homogeneous, and maps 0 to 0)  and satisfies that $\tau(x^*x)=\tau(xx^*)$ for all $x\in A$. We will always assume that traces are lower semicontinuous, i.e., such that $\tau(a)\leq \liminf_n \tau(a_n)$ if $a_n\to a$. We do  not assume,   however, that   traces are densely finite.  We denote the cone of all lower semicontinuous traces  by $\T(A)$. We prove below the following theorem:

\begin{theorem}\label{mainthm}
Let $A$ be a unital C*-algebra. Let  $a,b\in A$ be selfadjoint elements.
The following are equivalent:
\begin{enumerate}[(i)]
\item
$a\in \overline{\co\{ubu^*\mid u\in \U(A)\}}$,
\item
$\tau((a-t)_+) \leq \tau((b-t)_+)$ and  $\tau((-a-t)_+) \leq \tau((-b-t)_+)$
for all $\tau\in \T(A)$ and all $t\in \R$. 
\end{enumerate}
\end{theorem}
In this theorem $\co(\cdot)$ denotes the convex hull of a set and $\U(A)$ the unitary group of $A$. If (i) holds we say that $a$ is majorized by $b$. 
If $A$ is a simple C*-algebra with at least one bounded trace, then condition (ii) of Theorem \ref{mainthm} takes the following form, which is  closer to the matrix case: $\tau((a-t)_+)\leq \tau((b-t)_+)$ and $\tau(a)=\tau(b)$ for all bounded traces $\tau$ and all $t\in \R$ (Corollary \ref{simple} (i)). However, since we allow for traces that are  not densely finite Theorem \ref{mainthm} covers the simple purely infinite C*-algebras as well (Corollary \ref{simple} (ii)); indeed, it covers all C*-algebras.  A  related theorem, also valid for all C*-algebras, is 
\cite[Theorem 1.1]{robert2}, which shows that agreement of two positive elements on all traces in $\T(A)$ is equivalent to  the Cuntz-Pedersen relation. 

A few words on  the proof of Theorem \ref{mainthm}: We use a well-known Hahn-Banach argument going back to Day (\cite{day}) to reduce the proof  to the von Neumann algebra $A^{**}$.  In the von Neumann algebra setting, we deal first with finite von Neumann algebras  using arguments inspired by the II$_1$ factor case  and then extend the proof to the general case.  In the process we obtain a 
formula for the distance from $a$ to $\overline{\co\{ubu^*\mid u\in \U(A)\}}$
in terms of tracial inequalities  (the zero distance case of this formula is Theorem \ref{mainthm}). 

In the context of majorization of matrices  one can observe the following phenomenon:
For any given $\epsilon>0$ there exists $N\in \N$ such that if $a,b\in M_n(\C)$ are selfadjoint matrices of norm at most 1
and $a$ is majorized by $b$, then  there exists  a  convex combination of at most $N$ unitary conjugates
of $b$ which is within a distance of $\epsilon$ from $a$. Here the number $N$  does not depend on $a$ or $b$, as long as they are contractions, or on the matrix size  $n$ (see \cite[Theorem 6.1]{ng-skoufranis} for an explicit formula).  We refer to this property as uniform majorization. (In the language of continuous logic of C*-algebras, that $N$ depends solely on $\epsilon$ implies that the relation of majorization is uniformly definable within the class of matrix C*-algebras; see \cite{munster}.)     Uniform  majorization does not hold for general C*-algebras and may fail even in a single C*-algebra. We show below that the C*-algebra constructed in \cite[Theorem 1.4]{robert} does not have uniform majorization (Example \ref{nouniformmaj}). 
This C*-algebra, which is simple and  nuclear, fails to have various regularity
properties of great significance in the classification of simple nuclear C*-algebras; to wit, it neither has strict comparison of positive elements by traces nor finite nuclear dimension.  We prove below that  these very same regularity properties serve to ensure uniform majorization:

\begin{theorem}\label{uniformthm}
For every $\epsilon>0$ there exists $N\in \N$ such that if $A$ is a unital C*-algebra with strict comparison
of positive elements by traces and $a,b\in A$ are selfadjoint contractions such that 
$
a\in \overline{\co\{ubu^*\mid u\in \U(A)\}}
$  
then 
\[
\Big\|a-\frac 1 N\sum_{i=1}^N u_ibu_i^*\Big\|<\epsilon
\]
for some $u_1,\dots,u_N\in \U(A)$.
\end{theorem}

A version of Theorem \ref{uniformthm} for C*-algebras with finite nuclear dimension is also valid (Theorem \ref{uniformnuc}).  
We obtain the following interesting application of 
uniform majorization: 
Let $A$ be a unital C*-algebra with either strict comparison by traces or finite nuclear dimension.
Let  $B\subseteq A_{\infty}$ be a separable C*-subalgebra of the sequence algebra $A_\infty:=\prod_{i=1}^\infty A/\bigoplus_{i=1}^\infty A$. Then for every selfadjoint $a\in A_{\infty}$ the set   $\overline{\co(\{uau^*\mid u\in \U(A_\infty)\})}$ has non-empty intersection with 
$B'\cap A_{\infty}$.

The paper is organized as follows: In Section \ref{prelims} we define the majorization and submajorization relations and prove some of their general properties which will be needed later on. In Section \ref{vNcase} we prove  Theorem \ref{mainthm}
when $A$ is a von Neumann algebra (at this point  we assume that $a$ and $b$ are positive contractions as  matter of convenience). In Section \ref{proofofmain}
we prove Theorem \ref{mainthm} together with a more general distance formula and we derive some corollaries of these theorems. In Section \ref{uniform} we investigate the property of uniform majorization described above. The proof of Theorem \ref{uniformthm}, unlike the more hands-on methods used in \cite{ng-skoufranis}, does not yield and explicit formula for the number $N$ in terms of $\epsilon$.

\section{Preliminaries on majorization and submajorization}\label{prelims}
Let $A$ be a C*-algebra. Let us denote by $A_+$ and $A_{\sa}$  the sets of  positive and selfadjoint elements of $A$, respectively. If $A$ is unital, we let $\U(A)$ denote the unitary group of $A$.
If $a\in A_{\sa}$ and $t\in \R$  we denote by  $(a-t)_+$  the element obtained from $a$ by functional calculus  with the function $x\mapsto (x-t)_+:=\max(x-t,0)$.

Given  $a,b\in A_{\sa}$ let us say that $a$ is submajorized by $b$, and denote it by $a\smaj b$, if
\[
a\in\overline{\co(\{dbd^*\mid \|d\|\leq 1\})}.
\] 
Suppose that  $A$ is unital. Let us say that $a$ is majorized by $b$, and denote it by $a\maj b$, if 
\[
a\in \overline{\co(\{uau^*\mid u\in \mathrm U(A)\})}.
\]
It is possible to extend  the relation of majorization to non-unital C*-algebras simply by passing to the unitization. However, we will always assume that $A$ is unital when discussing majorization. 
Both submajorization and majorization are preorder relations. 

We use the following lemma quite frequently and without  reference:

\begin{lemma}\label{orthosum}
Let $a_1,a_2,b_1,b_2\in A_{\sa}$ be such that $a_1\smaj b_1$, $a_2\smaj b_2$,
$a_1a_2=0$ and $b_1b_2=0$. Then $a_1+a_2\smaj b_1+b_2$.
\end{lemma}
\begin{proof}
Let $\epsilon>0$. Suppose that 
\[
\|a_1-\frac 1 N\sum_{i=1}^N d_{i,1}b_1d_{i,1}^*\| <\epsilon\hbox{ and }
\|a_2-\frac 1 N\sum_{i=1}^N d_{i,2}b_2d_{i,2}^*\| <\epsilon,
\]
for some contractions $d_{i,1},d_{i,2}\in A$.
Multiplying by an approximate unit of $\overline{|a_1|A|a_1|}$ on the left and on the right of the first equation and replacing $b_1$ by $|b_1|^{\frac 1 n}b_1|b_1|^{\frac1 n}$ for large enough $n$ we can assume that $d_{i,1}\in \overline{|a_1|A|b_1|}$ for all $i$. Similarly, we can assume that $d_{i,2}\in  \overline{|a_2|A|b_2|}$ for all $i$.
 Define $d_{i}=d_{i,1}+d_{i,2}$ for all $i$.
A straightforward calculation exploiting that $a_1a_2=b_1b_2=0$ shows that
\[
\|(a_1+a_2)-\frac 1 N\sum_{i=1}^N d_{i}(b_1+b_2)d_{i}^*\|<2\epsilon.
\]
This proves the lemma.
\end{proof}

\begin{lemma}\label{basicsmaj}
	Let $a,b\in A_{\sa}$.
	\begin{enumerate}[(i)]
		\item
		If $a\leq b$, then $a_+\smaj b_+$.
		
		\item
		If   $\|a-b\|\leq r$, then $(a-r)_+\smaj b_+$.
		
		\item
		If  $a\smaj b$, then $(a-t)_+\smaj (b-t)_+$ for all $t\in [0,\infty)$.
	\end{enumerate}
\end{lemma}
\begin{proof}
	(i) Assume first that $a\geq 0$ (so $b\geq 0$).  Since $a\leq b$,  $a$ is in the hereditary C*-subalgebra generated by $b$. Hence,  $b^{\frac 1 n}ab^{\frac 1 n}\to a$, which  shows that $a\smaj b$, as desired.  Suppose now that  $a\in A_{\sa}$. Let $\epsilon>0$. Let $c\in C^*(a)$ be a positive contraction such that $ca=(a-\epsilon)_+$.  Multiplying by $c^{1/2}$ on the left and on the right of  $a\leq b$  we get 
	\[
	(a-\epsilon)_+\leq c^{\frac 1 2}bc^{\frac 1 2}\leq c^{\frac 1 2}b_+c^{\frac 1 2}\smaj 
	b_+.
	\] 
Since submajorization is transitive and we have already shown that the order on positive elements  is stronger than the submajorization relation,  $(a-\epsilon)_+\smaj b_+$ for all $\epsilon>0$.
	Letting $\epsilon\to 0$ we are done.
	
	(ii) We have that $a-r\leq b$. So we can apply (i) to get that $(a-r)_+\smaj b_+$.
	
	(iii)   Choose $b'=\frac 1 N\sum_{i=1}^N d_ibd_i^*$, with $\|d_i\|\leq 1$ for all $i$, such that
	$\|a-b'\|<\epsilon$. From $a-t-\epsilon\leq b'-t$ we get, by (i), that $(a-t-\epsilon)_+\smaj (b'-t)_+$. Also,
	\[
	b'-t\leq \frac 1 N\sum_{i=1}^Nd_i(b-t)d_i^*\leq  \frac 1 N\sum_{i=1}^Nd_i(b-t)_+d_i^*.
	\] 
	Hence, by (i), $(b'-t)_+$ is submajorized by $\frac 1 N\sum_{i=1}^Nd_i(b-t)_+d_i^*$, which in turn is submajorized by $(b-t)_+$. By the transitivity of submajorization, $(a-t-\epsilon)_+\smaj (b-t)_+$ for all $\epsilon>0$, from which the desired result follows.
\end{proof}

\begin{proposition}\label{reducetopositive}
	Let $a,b\in A_{\sa}$. Then $a\smaj b$ if and only if $a_+\smaj b_+$ and $a_-\smaj b_-$.
\end{proposition}	

\begin{proof}
Suppose first that $a\smaj b$. Let $a_n\in A_{\sa}$ be elements such that
$a_n\to a$ and each $a_n$ is a finite convex combination of elements 
of the form $dbd^*$. Since $(a_n)_+\to a_+$ it suffices to show that $(a_n)_+\smaj b_+$ for all $n$.  Put differently, it suffices to assume that $a=\frac 1 N\sum_{i=1}^Nd_ibd_i^*$ for some $\|d_i\|\leq 1$. In this case we have that  
\[
a\leq \frac 1 N\sum_{i=1}^Nd_ib_+d_i^*.
\] 
By Lemma \ref{basicsmaj},
\[
a_+\smaj (\frac 1 N\sum_{i=1}^Nd_ib_+d_i^*)_+=\frac 1 N\sum_{i=1}^Nd_ib_+d_i^*.
\]
The rightmost side is clearly submajorized by $b_+$. Thus, $a_+\smaj b_+$.
Since $-a\smaj -b$ we also have that $a_-=(-a)_+\smaj (b_+)=b_-$. This proves one implication.

Suppose now that $a_+\smaj b_+$ and $a_-\smaj b_-$. By Lemma \ref{orthosum}
we have that $a_+-a_-\smaj b_+-b_-$, i.e., $a\smaj b$, as desired.
\end{proof}

In light of the previous proposition we will largely focus on the study of the submajorization relation among positive elements. It will be easy enough to
extend our main results  to selfadjoint elements relying on  this proposition.

We call trace on $A$ a map $\tau\colon A_+\to [0,\infty]$ that is $\R^+$-linear, maps $0$ to $0$, and satisfies that $\tau(x^*x)=\tau(xx^*)$ for all $x\in A$. Notice that $\infty$ is in the range of $\tau$  and that we do not assume that $\tau$ is densely finite. We denote by $\T(A)$ the cone of all lower semicontinuous traces on $A$.  The reader is referred to \cite{elliott-robert-santiago} for basic facts on $\T(A)$.
Observe that for each closed two-sided ideal $I\subseteq A$ the map $\tau_I\colon A_+\to [0,\infty]$ defined as 
$\tau_I(a)=0$ if $a\in I_+$ and $\tau_I(a)=\infty$ otherwise is a lower semicontinuous trace. 
In particular, if we choose $I=\{0\}$ we get a trace that is $\infty$ everywhere except at
0. 

Let $a,b\in A_+$. We say that $a$ is tracially submajorized by $b$ if 
\[
\tau((a-t)_+)\leq \tau((b-t)_+)\hbox{ for all  $\tau\in \T(A)$ and all $t\in [0,\infty)$.}
\]
We denote this relation by $a\smajtau b$.

The following proposition clarifies the meaning of tracial submajorization in C*-algebras with ``very few'' traces.

\begin{proposition}
	Suppose that the C*-algebra $A$ has no l.s.c. traces other than the traces $\tau_I$ associated to its closed two-sided ideals (e.g., $A$ is purely infinite). Let $a,b\in A_+$.
	Then $a\smajtau b$ if and only if $\|\pi_I(a)\|\leq \|\pi_I(b)\|$ for all  quotient maps 
	$\pi_I\colon A\to A/I$.
\end{proposition}
\begin{proof}
	Let $I$ be a closed two-sided ideal of $A$. Denote by $\pi_I\colon A\to A/I$ the quotient map. 
	Let $t\in [0,\infty)$. Consider the inequality $\tau_I((a-t)_+)\leq \tau_I((b-t)_+)$. The right side is $\infty$ for all $t<\|\pi_I(b)\|$. So in this case
	the inequality is trivially  valid. On the other hand, if $t\geq \|\pi_I(b)\|$, then the inequality is valid if and only if the left side is 0, i.e., if $(a-\|\pi(b)\|)_+\in I$.  This is equivalent to  $\|\pi_I(a)\|\leq \|\pi_I(b)\|$, as desired. 
\end{proof}

We will show below that in any C*-algebra  tracial submajorization is equivalent to submajorization (for positive elements) but this will entail first    elucidating  independently  some of the properties of both relations.

\begin{lemma}\label{hereditary}
	Let $B\subseteq A$ be a hereditary C*-subalgebra. Let $a,b\in B_+$.
	\begin{enumerate}[(i)]
		\item
		If $a\smaj b$ in $A$, then $a\smaj b$ in $B$.
		\item
		If $a\smajtau  b$ in $A$, then $a\smajtau b$ in $B$.
	\end{enumerate}
\end{lemma}
\begin{proof}
	(i)	Let $(e_\lambda)_\lambda$ be an approximate unit of $B$ consisting of contractions.  Let $\epsilon>0$. 
	Say $d_1,\dots,d_N\in A$ are contractions such that 
	\[
	\Big\|a-\frac1 N\sum_{i=1}^N d_ibd_i^*\Big\|<\epsilon.\] 
	Call the left side  of the above inequality $\epsilon'$ and choose $\epsilon'<\epsilon''<\epsilon$. 
	We have
	\begin{align*}
\Big	\|a-\frac1 N\sum_{i=1}^N e_\lambda d_ibd_i^*e_\lambda\Big\|
	&\leq \|a-e_\lambda a e_\lambda\|+\Big\|e_\lambda 
	\Big(a-\frac1 N\sum_{i=1}^N d_ibd_i^*
	\Big)e_\lambda\Big\|\\
	&\leq \|a-e_\lambda a e_\lambda\|+\epsilon'.
	\end{align*}
Since $e_\lambda ae_\lambda\to a$, there exists $\lambda_0$ such that the left side is less than $\epsilon''$ 
	for all $\lambda\geq \lambda_0$.  Moreover, since $e_\lambda be_\lambda\to b$, there exists $\lambda_1$ such that
	\[
	\|a -\frac1 N\sum_{i=1}^N e_\lambda d_i(e_\lambda be_\lambda)d_i^*e_\lambda\|< \epsilon
	\]
	for all $\lambda\geq \lambda_1$. Notice that $e_\lambda d_ie_\lambda\in B$ for all $i$. Thus,  $a\smaj b$ in $B$.

	(ii) It suffices to show that every l.s.c. trace on $B$ extends to $A$. 
	Let us sketch the proof of this  known fact: Given positive elements  $e,f\in A_+$ let us write $e\precsim_{\mathrm{CP}}f$ if
	$e=\sum_{i=1}^\infty x_i^*x_i$ and $\sum_{i=1}^\infty x_ix_i^*\leq f$ for some $x_i\in A$, where the series are convergent in norm. This  transitive relation is studied in \cite{CuntzPedersen} and \cite{robert2}.   To define an extension of a trace $\tau$ on $B$ to $A$ we set    
	\[
	\tilde\tau(x)=\sup\{\tau(y)\mid  y\in B_+,\, y \precsim_{\mathrm{CP}} x\},
	\]
	for all $x\in A_+$. 
	Then $\tilde \tau$ is a l.s.c. trace on $A$ extending $\tau$. The proof of this claim may be found in the proof of \cite[Lemma 4.6]{CuntzPedersen}.
\end{proof}

Let $\mathcal K$ denote the C*-algebra of compact operators on a separable, infinite dimensional, Hilbert space.
We regard $A$ embedded in $A\otimes \mathcal K$ in the usual manner, i.e., by placing  the elements of $A$  in the upper left-corner of an infinite matrix whose entries are 0 everywhere else.

\begin{proposition}\label{smajstabilization}
	Let $a,b\in A_+$. Then $a\smaj b$ in $A$ if and only if 
	$a\maj b$ in  $(A\otimes \mathcal K)^\sim$ (i.e., in the unitization of the stabilization of $A$). 
\end{proposition}
\begin{proof}
	Suppose that  $a\maj b$ in $(A\otimes \mathcal K)^\sim$. Since $A$
	is a hereditary C*-subalgebra of $(A\otimes \mathcal K)^\sim$, we have $a\smaj b$
	in $A$ by Lemma \ref{hereditary} (i).

	Let us prove the opposite implication. We consider first the case when $a\leq b$.	Let $n\in \N$. We have
	\[
	b=
	\begin{pmatrix}
	a^{1/2}  &\cdots  &(b-a)^{1/2}
	\end{pmatrix}
	\begin{pmatrix}
	a^{1/2}\\
	\vdots\\
	(b-a)^{1/2}
	\end{pmatrix}
	\]	
	and 	
	\[
	\begin{pmatrix}
	a^{1/2}\\
	\vdots\\
	(b-a)^{1/2}
	\end{pmatrix}
	\begin{pmatrix}
	a^{1/2}  &\cdots &(b-a)^{1/2}
	\end{pmatrix}
	=		\begin{pmatrix}
	a & \cdots & a^{1/2}(b-a)^{1/2}\\
	\vdots & &\vdots\\
	(b-a)^{1/2}a^{1/2}& \cdots& b-a
	\end{pmatrix}\in M_n(A),
	\]	
	where the omitted entries are  all zeros.
	By changing $n$  and averaging we find that for any $\epsilon>0$ we can choose $x_1,\dots,x_N\in A\otimes \mathcal K$ such that
	\begin{equation*}
	 \|a-\frac 1 N\sum_{i=1}^Nx_ix_i^*\|<\epsilon \hbox{ and }  b=x_i^*x_i\hbox{ for all }i.
	\end{equation*}
 But for all $x\in A\otimes\mathcal K$ the elements
	$x^*x$ and $xx^*$ are approximately unitarily equivalent in $(A\otimes\mathcal K)^\sim$ (\cite[Lemma 4.3.3]{radius}). This shows that $a\maj b$ in $(A\otimes\mathcal K)^\sim$, as desired.
	
	Suppose now that $a=dbd^*$, with $\|d\|\leq 1$. Let $x=db^{1/2}$. Then $a=xx^*$ and $x^*x\leq b$. We have already shown that $x^*x\maj b$ in $(A\otimes\mathcal K)^\sim$.
	But, as remarked above,  $x^*x$ and $xx^*$ are approximately unitarily equivalent in 
	$(A\otimes\mathcal K)^\sim$. So $a=x^*x\maj xx^*\maj b$.
	
	Consider the general case. Suppose that $a\smaj b$. Then $a$ is a limit of convex combinations of elements of the form $dbd^*$, with $\|d\|\leq 1$. We have already shown that
	each of these elements is majorized by $b$ in $(A\otimes\mathcal K)^\sim$. It follows that $a\maj b$ in $(A\otimes\mathcal K)^\sim$, as desired.
\end{proof}

\begin{proposition}\label{expectation}
	Let $E\colon A\to A$ be a positive contractive map that is also trace decreasing, i.e.,
	$\tau(E(a))\leq \tau (a)$ for all $\tau\in \T(A)$ and all $a \in A_+$. 
	Then $E(a)\smajtau a$ for all $a\in A_+$.
\end{proposition}
\begin{proof}
	Let $t\in [0,\infty)$ and $\tau\in \T(A)$. Let $\epsilon>0$. 
	Since $E$ is positive and contractive we have that 
	\[
	E(a)-t\cdot 1\leq E(a-t\cdot 1)\leq E((a-t)_+).
	\] Let $c\in C^*(E(a))$ be a positive contraction   such that $ (E(a)-t)c=(E(a)-t-\epsilon)_+$. Then 
	\[
	(E(a)-t-\epsilon)_+\leq c^{\frac 1 2}E((a-t)_+)c^{\frac 1 2}.
	\]
Evaluating both sides on  $\tau$  and using that $E$ is trace decreasing we get that
	\[
	\tau((E(a)-t-\epsilon)_+))\leq \tau(E((a-t)_+))\leq \tau((a-t)_+).
	\]
	Letting $\epsilon \to 0$ and using that $\tau$ is lower semicontinuous
	we get the desired inequality.
\end{proof}

\begin{proposition}\label{easyimplication}
	Let $a,b\in A_+$.
		If $a\smaj b$,  then $a\smajtau b$.
\end{proposition}
\begin{proof}
 First suppose that $a$ is exactly a convex combination of elements of the form $dbd^*$, with $\|d\|\leq 1$. Say $a=\sum_{i=1}^n t_i d_ibd_i^*$,
where $\|d_i\|\leq 1$ for all $i$, $0\leq t_i\leq 1$ for all $i$, and $\sum_{i=1}^n t_i=1$.
	Let $E\colon A\to A$ be defined as $E(y)=\sum_{i=1}^n t_i d_iyd_i^*$ for all $y\in A$. Then $E$ is positive, contractive, and trace decreasing. By  Proposition \ref{expectation}, 
	$a=E(b)\smajtau b$ as desired.
	
	Suppose now that $a,b\in A_+$ are arbitrary elements such that $a\smaj b$. Let $a_n\to a$ where $a_n$ is a finite convex combination of elements of the form
	$dbd^*$, with $\|d\|\leq 1$. Then $a_n\smajtau b$ for all $n$ by the previous case. Let $\tau\in \T(A)$
	and $t\in [0,\infty)$. Then  $\tau((a_n-t)_+)\leq \tau((b-t)_+)$ for all $n$ and $(a_n-t)_+\to (a-t)_+$. By the lower semicontinuity  $\tau$,  
	\[
	\tau((a-t)_+)\leq \liminf_n \tau((a_n-t)_+)\leq \tau((b-t)_+),\] 
	as desired. (What we have shown is that the set of elements tracially submajorized by $b$ is closed.)
\end{proof}

\section{Von Neumann algebra case}\label{vNcase}
In this section we work exclusively in the setting of von Neumann algebras. 
The main results of this section, Propositions \ref{vNsubmajorization} and \ref{vNmajorization},  characterize submajorization and majorization in a von Neumann algebra in terms of tracial submajorization. They are stepping stones towards proving the same results for all C*-algebras. (We  take-up this task in the next section.)

Throughout this section $M$ denotes a von Neumann algebra. We also fix the following notations: The center of $M$ is denoted by $Z$.  Elements of $Z$
are often regarded as continuous functions on $\widehat Z$ (the spectrum of $Z$). Given $a\in M$ we denote by $c_a\in Z$  the central carrier or central support projection of $a$.

\begin{lemma}\label{centrallambda}
Let $a,b\in M_+$ be  such that $a\smajtau b$. Let $\lambda\in Z_+$.
	The following are true:
	\begin{enumerate}[(i)]
		\item
		$\lambda a\smajtau \lambda b$	
		\item
		$(a-\lambda)_+\smajtau (b-\lambda)_+$.
		\item
		$a+\lambda\smajtau b+\lambda$
	\end{enumerate}	
	
	\begin{proof}
		(i) Let $\tau\in \T(M)$ and $t\in [0,\infty)$. We must show that 
		$\tau((\lambda a-t)_+)\leq \tau((\lambda b-t)_+)$.
		Suppose first that $\lambda=e$  is a central projection. 
		Then $\tau((ea-t)_+)=\tau(e(a-t)_+)$, and since $x\mapsto \tau(e\cdot x)$ is a trace in $\T(M)$, $\tau(e(a-t)_+)\leq \tau(e(b-t)_+)$. So $\tau((ea-t)_+)\leq \tau((eb-t)_+)$, as desired.
		
		Suppose now that $\lambda$ has finite spectrum. 
		Then   $\lambda=\sum_{i=1}^n \alpha_ie_i$, where $e_1,\ldots,e_n$ are pairwise orthogonal central projections and where $\alpha_1,\ldots,\alpha_n>0$
		are  positive scalars. We have
		\begin{align*}
			\tau((\lambda a-t)_+)&=\tau\Big(\Big(\sum_{i=1}^n\alpha_ie_ia-t\Big)_+\Big)\\
			&=\tau\Big(\sum_{i=1}^n\alpha_i\Big(e_ia-\frac{t}{\alpha_i}\Big)_+\Big) 
			\leq \tau\Big(\sum_{i=1}^n\alpha_i\Big(e_ib-\frac{t}{\alpha_i}\Big)_+\Big)
			=\tau((\lambda b -t)_+).
		\end{align*}

		Finally, suppose that $\lambda$ is an arbitrary positive central element. Choose an increasing sequence of positive  central elements $(\lambda_n)_n$ 
		each with finite spectrum and  such that $\lambda_n \nearrow \lambda$ in norm. We have already proven that 
		$
		\tau((\lambda_na-t)_+)\leq \tau((\lambda_nb-t)_+)
		$
		for all $n$. Observe that $(\lambda_na-t)_+\nearrow (\lambda a-t)_+$ and 
		$(\lambda_nb-t)_+\nearrow (\lambda b-t)_+$. So passing to the limit as $n\to \infty$ and using that $\tau$ is l.s.c.
		we get that  $\tau((\lambda a-t)_+)\leq \tau((\lambda b-t)_+)$, as desired.
		
		(ii) It suffices to show that $\tau((a-\lambda)_+)\leq \tau((b-\lambda)_+)$ for all $\tau\in \T(M)$. Choose a decreasing sequence of positive central elements $(\lambda_n)_n$ with finite spectrum and such that $\lambda_n\searrow \lambda$ in norm. Then $(a-\lambda_n)_+\nearrow (a-\lambda)_+$ and $(b-\lambda_n)_+\nearrow (b-\lambda)_+$. So, arguing as in (i), the proof is reduced to the case of $\lambda$ with finite spectrum.

		Say $\lambda=\sum_{i=1}^n \alpha_ie_i$ where $e_1,\ldots,e_n$ are pairwise orthogonal central projections adding up to 1 and $\alpha_i\geq 0$ are scalars.
		Then 
		\[
		\tau(e_i(a-\lambda)_+)=\tau(e_i(a-\alpha_i)_+)\leq \tau(e_i(b-\alpha_i)_+)=\tau(e_i(b-\lambda)_+),
		\]
		for all $i$.	Adding over all  $i$  we get the result.
		
		(iii) We can reduce the proof to the case of a $\lambda$ with finite spectrum by choosing an increasing sequence $(\lambda_n)_n$ such that $\lambda_n\nearrow \lambda$ in norm and arguing as in (i).  Passing to  central cut-downs $e_iM$, where $e_1,\ldots,e_n$ are central projections adding up to $1$, we are further reduced to the case that $\lambda$ is a nonnegative scalar.
		So assume that this is the case. Then $(a+\lambda-t)_+=a+(\lambda-t)$ if $t\leq \lambda$ and 
		$(a+\lambda-t)_+=(a-(t-\lambda))_+$ otherwise. This calculation shows that $a\smajtau b$ implies that $a+\lambda\smajtau b+\lambda$.
	\end{proof}	
\end{lemma}

\begin{proposition}\label{theform}
	Let $a,b\in M$ be positive elements with finite spectrum. 
	Then 
	\begin{equation}\label{formab}
		a=\sum_{i=1}^n \alpha_iP_i,\quad
		b=\sum_{i=1}^{n} \beta_iQ_i
	\end{equation}
	for some $(P_i)_{i=1}^n$ and $(Q_i)_{i=1}^n$, pairwise
	orthogonal projections in $M$ adding up to $1$  such that $P_i\sim Q_i$ for all $i$,
	and some decreasing sequences of positive central elements $(\alpha_i)_{i=1}^n$ and  $(\beta_i)_{i=1}^n$  
	such that $\|\alpha_i\|\leq \|a\|$ and $\|\beta_i\|\leq \|b\|$ for all $i$.
\end{proposition}

\begin{proof}
	Since $a$ and $b$ have finite spectrum,
	we have decompositions
	\begin{equation}\label{startform}
		a = \sum_{i=1}^l \mu_i E_i,\quad
		b = \sum_{j=1}^m \nu_j F_j
	\end{equation}
	where $(E_i)_{i=1}^l$ and $(F_j)_{j=1}^m$ are pairwise orthogonal projections adding  up to 1, 
	and $(\alpha_i)_{i=1}^m$ and 
	$(\beta_j)_{j=1}^m$  are nonnegative  scalars.  We further assume that both sequences 
	have been arranged in decreasing order.

	We will prove the  representation for $a$ and $b$ in \eqref{formab} by induction on $l + m$.
	The base case is  $l + m = 2$, i.e., $l=m=1$.
	In this case both $a$ and $b$ are scalars multiples of the identity.
	The desired representation has already been achieved.
	
	Suppose that the desired representation  is true for all pairs $a$ and $b$ as in \eqref{startform}
	such that $l + m$ is less than a given number.
	Now suppose that $l + m$ is that given number.
	Observe that if $(e_k)_{k=1}^N$  are central projections adding up to 1 and 
	the desired representation has been obtained for $e_ka$ and $e_kb$ in 
	$e_kM$ for all $k$   then  adding up these representations---adding zero terms if necessary so that they have the same number of terms---we get the desired  representation   for $a$
	and $b$. Now recall  that there is
	a central projection $e$ such that $eE_1\precsim eF_1$ and $(1-e)F_1\precsim (1-e)E_1$ (\cite[Theorem 6.2.7]{KadRing}).
	Hence, reducing the proof to $eM$ and $(1-e)M$,  we can assume that $E_1$ and $F_1$ are Murray-von Neumann comparable. By symmetry, it suffices to assume that
	$E_1\precsim F_1$. Recall also that for any projection
	$P\in M$ there exists a central projection $e$ such that $eP$ is a finite projection and $(1-e)P$
	is properly infinite (\cite[Proposition 6.3.7]{KadRing}). Applying this to $E_1$, and  reducing the proof to each central cut-down, we can assume that $E_1$ is either finite or properly infinite. 
	
	\emph{Case 1: $E_1$  is finite}.  Let us find $F_1'\leq F_1$ such that $E_1\sim F_1'$.
	Since $E_1$ is finite,  there exists a unitary $u$ such that $uE_1u^*=F_1'$ (\cite[Exercise 6.9.7]{KadRing}). 
	Since it is sufficient to obtain the desired representation for $uau^*$ and $b$, let us rename $uau^*$
	as $a$ and assume that $E_1=F_1'$. Let 
	\[
	a'=\sum_{i=2}^l \mu_i E_i,\quad b'=\nu_1(F_1-F_1')+\sum_{j=2}^m \nu_jF_j. 
	\]
	Notice that the total number of projections supporting $a'$ and $b'$ is now $l+m-1$. We can thus apply the induction hypothesis in the von Neumann algebra $(1-F_1')M(1-F_1')$ to get
	\[
	a'=\sum_{i=1}^n \alpha_iP_i,\quad
	b'=\sum_{i=1}^n \beta_iQ_i.
	\]
	We  also have by induction that $\alpha_1\leq \|a'\|$ and $\beta_1\leq \|b'\|$.
	The map $x\mapsto (1-F_1')x$ is a surjective homomorphism from $Z$ to
the center of $(1-F_1')M(1-F_1')$ (by \cite[Theorem 5.4.1]{sinclair-smith}).
Thus, 	the decreasing central elements $(\alpha_i)_{i=1}^n$ and $(\beta_i)_{i=1}^n$ in $(1-F_1')M(1-F_1')$ can be lifted to central elements in $Z$. Moreover, as is clear for any surjective map between abelian von Neumann algebras,  the  decreasing order and the inequalities  $\alpha_1\leq \|a'\|$ and $\beta_1\leq \|b'\|$ can be maintained after this lifting. Let us continue to denote these central liftings 
	by $\alpha_i$ and $\beta_i$.  Notice that $\mu_1\geq \mu_2=\|a'\|\geq \alpha_i$ for all $i$ and $\nu_1\geq \|b'\|\geq \beta_i$ for all $i$. So 
	\[
	a=\mu_1E_1+\sum_{i=1}^n\alpha_iP_i,\quad
	b=\nu_1F_1'+\sum_{i=1}^n \beta_iQ_i
	\]
	are the desired representations for $a$ and $b$.
	
	\emph{Case 2: $E_1$   is properly infinite}. We can find a central projection $e$ such that
	$eE_1\precsim e(1-E_1)$ and $(1-e)(1-E_1)\precsim (1-e)E_1$. By passing to the corresponding central cut-down, we arrive at two cases: 
	
	\emph{Case 2(a): $E_1\precsim 1-E_1$}. Let us again find $F_1'\leq F_1$
	such that $E_1\sim F_1'$. Let us moreover choose $F_1'$ such that $F_1'\precsim 1-F_1'$. We can easily achieve this exploiting that $E_1$ is properly infinite. We claim that $1-E_1\sim 1\sim 1-F_1'$. Indeed, say
	$E_1'\leq 1-E_1$ is such that $E_1\sim E_1'$. Since we have assumed that
	$E_1$ is properly infinite, $E_1+E_1'\sim E_1$.
	Hence  
	\[
	1=(1-E_1-E_1')+E_1'+E_1\sim (1-E_1-E_1')+E_1'=1-E_1.
	\]
	We prove similarly that $1\sim 1-F_1'$, thereby establishing our claim. 
	From  $E_1\sim F_1'$ and   $1-E_1\sim 1-F_1'$ we again deduce---as in the case where $E_1$ is finite---that there exists a unitary $u$ such that $uE_1u^*=F_1'$. We can now continue arguing  as in the case where $E_1$ is finite to complete the induction step.
	
	\emph{Case 2(b): $1-E_1\precsim E_1$}. 
	Since $E_1$ is properly infinite, $E_1\sim 1$. (Proof:  We have  $E_1\leq 1$. 
	So, by Cantor-Bernstein, it suffices to show that $1\precsim E_1$. Indeed,
	\[
	1=(1-E_1)+E_1\precsim E_1\oplus E_1\sim E_1.)
	\] 
	Moreover, since $E_1\precsim F_1$, we have $F_1\sim 1$ as well. We can thus decompose $E_1$
	and $F_1$ as follows: $E_1=E_1'+E_1''$ and $F_1=F_1'+F_1''$, where $E_1',E_1'',F_1',F_1''$ 
	are projections  such that $E_1'\sim F_1'\sim 1$, $E_1''\sim 1-F_1$, and $F_1''\sim 1-E_1$.
	Notice that $E_1'\sim F_1'$ and that
	\[
	1-E_1'=(1-E_1)+E_1''\sim F_1''+(1-F_1)=1-F_1'.
	\]
	So  there exists a unitary $u$ such that $uE_1'u^*=F_1'$. It suffices to find the desired representations for $uau^*$ and $b$. Let us relabel $uau^*$ as $a$ and assume that 
	$E_1'=F_1'$. We have that
	\[
	a=\mu_1 F_1'+\mu_1 E_1''+\sum_{i=2}^l \mu_i E_i,
	\]
	while $b$ has the form
	\[
	b=\nu_1 F_1'+\nu_1 F_1''+\sum_{j=2}^m \nu_j F_j.
	\]
	It is thus clear that it suffices to find the desired representations for 
	\[
	a'=\mu_1E_1''+\sum_{i=2}^l \mu_i E_i,\quad
	b'=\nu_1F_1''+\sum_{j=2}^m \nu_j F_j
	\] 
	in the von Neumann algebra  $(1-F_1')M(1-F_1')$ and then lift the central coefficients to $M$
	(as in Case 1 above). Notice that the number of projections supporting $a'$ and $b'$ is still $l+m$. However, repeating the arguments used above we will find ourselves in either Case 1 or Case 2(a). More specifically, working in the von Neumann algebra $(1-F_1')M(1-F_1')$, we can find central projections $e_1,e_2,e_3,e_4$ adding up to the unit $1-F_1'$  and  such that  
\begin{enumerate}[(1)]
	\item
	 either  $e_iE_1''\precsim e_iF_1''$ or  $e_iF_1''\precsim e_iE_1''$ for all $i=1,2,3,4$,
\item
 $e_iE_1''$ is either finite or properly infinite for all $i=1,2,3,4$.
 \end{enumerate}
 Passing to the algebra $e_i(1-F_1')M(1-F_1')$, let us assume first that
	$E_1''\precsim F_1''$. 
	Then 
	\[
	E_1''\precsim F_1''\sim 1-E_1=(1-E_1')-E_1'' = (1-F_1')-E_1''.
	\] 
	So we can continue arguing as in Cases 1 and 2(a). Similarly, if $F_1''\precsim E_1''$, then 
	$F_1''\precsim (1-F_1')-F_1''$ so again we can continue arguing as in Cases 1 and 2(a).
	This completes the induction.
\end{proof}

\begin{lemma}\label{tracialineq1}
	Let $a,b\in M_+$ be  positive elements with finite spectrum represented  as in \eqref{formab} of Proposition \ref{theform}.
	If $a\smajtau b$, then 
	\begin{equation}\label{partialsmaj}
		\sum_{i=1}^k \alpha_iP_i\smajtau \sum_{i=1}^k \beta_iQ_i
	\end{equation}
	for $k=1,\dots,n$. 
\end{lemma}

\begin{proof}
	Since $P_i\sim Q_i$ for all $i$ and both sets of projections add up to 1, there exists a unitary $u$ such that $uQ_iu^*=P_i$ for all $i$. 
	Let us relabel $ubu^*$ as $b$  and assume that $P_i=Q_i$ for all $i$.
	
	We prove the lemma by induction on $k$. 
	Let us first prove that $\alpha_1c_{P_1}\leq \beta_1c_{P_1}$, which clearly implies the case $k=1$. 
	Passing to the central cut-down $c_{P_1}M$ if necessary we may assume that $c_{P_1}=1$ (since $a\smajtau b$ implies that $ac_{P_1}\smajtau bc_{P_1}$ in $c_{P_1}M$.) Suppose for the sake of contradiction that $\alpha_1\nleq \beta_1$. Then there exists a  projection $e\in Z$ and a scalar $\epsilon>0$
	such that $\alpha_1e\geq \beta_1e+\epsilon e$.  Since the central coefficients 
	$(\alpha_i)_{i=1}^n$ and $(\beta_i)_{i=1}^n$ are decreasing, we deduce 
	that $\|ea\|>\|eb\|$. But  this contradicts that $ea\smajtau eb$. Therefore,  $\alpha_1\leq \beta_1$.
	
	Suppose that the lemma is  true for $k-1$. 
	To prove \eqref{partialsmaj} it suffices to do it on each central cut-down $e_iM$ of a partition of unity by central projections $e_1,\ldots,e_N$. Since $Z$ is an abelian von Neumann algebra, given any two   positive elements $\alpha,\beta\in Z$ it is possible to find a   projection $e\in Z$ such that $e\alpha\geq e\beta$ and $(1-e)\alpha\leq (1-e)\beta$. Thus, we can reduce the proof to two cases: $\alpha_k\geq \beta_k$ or $	\alpha_k\leq \beta_k$. The second  case follows at once from the induction hypothesis.
	Let us assume that $\alpha_k\geq \beta_k$. We have that $(a-\beta_{k+1})_+ \smajtau (b -\beta_{k+1})_+$, by Lemma \ref{centrallambda} (ii). Hence
	\begin{align*}
		\sum_{i=1}^k (\alpha_i-\beta_{k+1})P_i &\leq \sum_{i=1}^{n}(\alpha_i-\beta_{k+1})_+P_i\\
&=(a-\beta_{k+1})_+ \smajtau  (b-\beta_{k+1})_+= \sum_{i=1}^{k}(\beta_i-\beta_{k+1})P_i.
	\end{align*}
	The above tracial submajorization  holds in the hereditary subalgebra $(P_1+\cdots +P_k)M(P_1+\cdots +P_k)$ (by Lemma \ref{hereditary} (ii)). Since 
	$\beta_{k+1}(P_1+\cdots P_{k})$ is a central element of this von Neumann algebra, we can add it  on both sides 
	by Lemma \ref{centrallambda} (iii). This yields  \eqref{partialsmaj}. 
\end{proof}		

\begin{lemma}
	Let $a,b,c\in M_+$ be  such that $a\maj c$ and $b\maj c$. Then for any central element $0\leq \lambda \leq 1$ we have that $\lambda a+(1-\lambda)b\maj c$.
\end{lemma}
\begin{proof}
	By a simple limiting argument it suffices to consider the case that $\lambda$ has finite spectrum.
	Say $\lambda=\sum_{i=1}^n \alpha_ie_i$ where $e_1,\ldots,e_n$ are pairwise orthogonal central projections adding up to 1 and $\alpha_i\in [0,1]$ for all $i$. In order  to show  that $\lambda a+(1-\lambda)b\maj c$ it suffices to show that
	$
	e_i(\alpha_i a+(1-\alpha_i)b)\maj e_ic
	$ in $e_iM$ for all $i$. But $e_ia,e_ib\maj e_ic$ for all $i$ and $e_i\alpha_i a+e_i(1-\alpha_i)b$ is a scalar convex combination of $e_ia$ and $e_ib$. The lemma is thus proved.
\end{proof}

\begin{lemma}\label{pinching}Let   $P,Q\in M$ be orthogonal projections
	and $\mu,\nu\in Z_+$. 
	There exists $\rho\in Z_+$ such that $\min(\mu,\nu)\leq \rho\leq \max(\mu,\nu)$   and such that for any central element $0\leq \lambda\leq 1$ we have
	\[
	\mu'P+\nu'Q\maj \mu P+\nu Q,
	\] 
	where
	\begin{align*}
		\mu' &=\mu\lambda+(1-\lambda)\rho,\\
		\nu'&=\nu\lambda+(1-\lambda)\rho.
	\end{align*}
\end{lemma}
\begin{proof}
	By Dixmier's approximation theorem (\cite[Theorem 8.3.5]{KadRing}) applied in the von Neumann algebra  $(P+Q)M(P+Q)$ we have that
	\[
	\rho P+\rho Q\maj  \mu P+\nu Q,
	\]
	for some $\min(\mu,\nu)\leq \rho\leq \max(\mu,\nu)$ in the center of $(P+Q)M(P+Q)$. We can lift $\rho$
	to an element in the center of $M$ satisfying the same inequalities. Let  $\lambda\in Z$ be such that $0\leq \lambda\leq 1$. Then, by the previous lemma,
	\[
	(\mu \lambda+\rho(1-\lambda))P+(\nu \lambda+\rho(1-\lambda))Q\maj \mu P+\nu Q,
	\]
	as desired.
\end{proof}

\begin{remark}  
	In the case that $M$ is finite one can show that $\rho=\frac{\mu E(P)+\nu E(Q)}{E(P+Q)}$, where $E\colon M\to Z$ is the center-valued trace.
\end{remark}

In the following proposition we assume that  $M$ is a finite von Neumann algebra. We denote by   $E\colon M\to Z$ the center valued trace of $M$.

\begin{proposition}\label{finitealg}
	Suppose that $M$ is a finite von Neumann algebra. Let $a,b\in M_+$ be positive elements of the form \eqref{formab} in Proposition \ref{theform}.  
	If
	\begin{equation}\label{finitecondition}
		\sum_{i=1}^k \alpha_iE(P_i)\leq \sum_{i=1}^k\beta_i E(Q_i)
	\end{equation}
	for all $k=1,\dots,n$, then $a\smaj b$.
\end{proposition}	
\begin{proof}
	Conjugating $b$ by a unitary	we may assume that $P_i=Q_i$ for all $i$.
	Passing to central cut-downs $e_jM$ for suitable projections $e_1,\ldots,e_N\in Z$ that partition the unit we may assume that $c_{P_i}=1$ for all $i$.
	Assuming these simplifications, we prove the proposition   by induction on $n$. More specifically, we  will show by induction on $n$ that if $P_1,\ldots,P_n$ are pairwise orthogonal projections
	in a finite von Neumann algebra such that $c_{P_i}=1$ for all $i$, and 
	$(\alpha_i)_{i=1}^n$ and $(\beta_i)_{i=1}^n$ are decreasing positive central elements such that  
	\begin{equation}\label{finitecondition2}
	\sum_{i=1}^k \alpha_iE(P_i)\leq \sum_{i=1}^k\beta_i E(P_i) \hbox{ for all }k=1,\ldots,n,
	\end{equation}
	then $a=\sum_{i=1}^n\alpha_iP_i\smaj \sum_{i=1}^n\beta_iP_i=b$. We do not assume (mostly as a matter of convenience), that the projections $P_i$ add up to 1.

	Consider the case $n=1$. From the inequality  \eqref{finitecondition2} we get that
	$\alpha_1E(P_1)\leq \beta_1E(P_1)$.
	Since  $c_{P_1}=1$ this implies that $\alpha_1\le \beta_1$, which in turn implies that $a\leq b$. By Lemma \ref{basicsmaj} (i), $a\smaj b$ as desired.
	
	Suppose now, by induction, that the desired result   is valid whenever the number of projections is less than $n$.
Consider the case of $n$ projections. 
	Let us apply Lemma \ref{pinching} to $\beta_1 P_1+\beta_2P_2$ with a suitable $0\leq \lambda\leq 1$ (to be specified soon)
	so as to obtain  $\rho\in Z$ and $\beta_1'P_1+\beta_2'P_2$ majorized by $\beta_1 P_1+\beta_2P_2$. Since $\beta_1\geq\beta_2$
	we have $\beta_1\geq \rho\geq \beta_2$ and $\beta_1'\geq \beta_2'$. Let us choose $\lambda$ such that the $\widehat Z $ (the spectrum of $Z$)
	partitions into two clopen sets satisfying that  
	\begin{enumerate}
		\item[(C1)]	
		$\beta_1'=\alpha_1$ on the first set
		\item[(C2)]
		$\beta_1'=\beta_2'\geq \alpha_1$ on the second set
	\end{enumerate}
	To see that this is possible, notice that the inequality $\beta_1'\geq \alpha_1$, put in terms of $\lambda$, has the   form
	\[
	\kappa \lambda\geq \gamma,
	\]
	for some $\kappa\in Z_+$ and some $\gamma\in Z_{\sa}$ such that 
	$\kappa\geq \gamma$  (in fact,  $\kappa=\beta_1-\rho\in Z_+$ and $\gamma=\alpha_1-\rho\in Z_{\sa}$). Let us choose $\lambda=\gamma_+/\kappa$, where the fraction is defined to be zero outside the set $\overline{\{x\in \widehat Z\mid \kappa(x)>0\}}$. (Recall that we regard  elements of $Z$ as continuous functions on its spectrum $\widehat Z$). The quotient $\gamma_+/\kappa$ is well defined in $Z$ since $\kappa\geq \gamma_+$ and  $Z$ is an abelian von Neumann algebra. Observe that $0\leq \lambda\leq 1$. Let us partition  $\widehat Z$   into the sets  $\overline{\{x\in \widehat Z\mid \lambda(x)>0\}}$ and its complement. These sets are clopen since $\widehat Z$ is  extremally disconnected. On the first set we have that $\kappa\lambda=\gamma$, which, put back in terms of $\beta_1'$, implies that $\beta_1'=\alpha_1$. Thus, we are in case (C1) above. On the second set we have that $\lambda=0$. This implies that  $\beta_1'=\beta_2'\geq \alpha_1$; i.e,  we are in case (C2). Thus, $\lambda$ is as desired.
	
	Let \[
	b'=\beta_1'P_1+\beta_2'P_2+\sum_{i>2}\beta_iP_i.\] 
	Then \eqref{finitecondition2} continues to hold for $a$ and $b'$. Indeed, for $k=1$ because $\beta_1'\geq \alpha_1$, and for
	$k>1$ because
	\[
	\beta_1'E(P_1)+\beta_2'E(P_2)= \beta_1E(P_1)+\beta_2E(P_2).
	\]
	Since $b'\maj b$, in order to prove the proposition it suffices to  show that $a\smaj b'$. So let us rename $b'$ as $b$, $\beta_1'$ as $\beta_1$, and $\beta_2'$
	as   $\beta_2$. 
	
	We can restrict to the two clopen sets described above and prove the proposition in each case.  (In other words, if $e_1,e_2\in Z$ are the central projections corresponding to these sets, then $e_1a$ and $e_1b$ continue to satisfy \eqref{finitecondition2} in $e_1M$ (keep in mind that the center valued trace of $e_1M$ is $e_1E(\cdot)$) and similarly for $e_2a$ and $e_2b$ in $e_2M$. Moreover, it suffices to show that $e_ia\smaj e_i b$ in $e_iM$ for $i=1,2$.) We claim that after restricting to the first set
	we are done by induction. Indeed, from \eqref{finitecondition2}, and keeping in mind that $\beta_1=\alpha_1$ on this set, we obtain that
	the elements
	\[
	a'=\sum_{i=2}^n \alpha_iP_i\hbox{ and }b''=\sum_{i=2}^n \beta_iP_i
	\]
	satisfy the induction hypothesis. So $a'\smaj b''$.  By Lemma \ref{hereditary} this   relation holds in    the hereditary subalgebra $(P_2+\cdots+P_n)M(P_2+\cdots+P_n)$.
	Therefore,  
	\[
	a=\alpha_1P_1+a'\smaj \beta_1P_1+b''=b,\] 
	as desired.

	Let us restrict to the second set where $\beta_1=\beta_2\geq \alpha_1$. Suppose more generally that for some $1<k\leq n$ we have that 
	$\beta_1=\cdots=\beta_k\geq \alpha_1$.   Assume first that $k<n$. Let us apply Lemma \ref{pinching} to $\beta_1(P_1+\cdots+P_k)+\beta_{k+1}P_{k+1}$ yielding
	the element $\beta_1'(P_1+\cdots+P_k)+\beta_{k+1}'P_{k+1}$ majorized by $\beta_1(P_1+\cdots+P_k)+\beta_{k+1}P_{k+1}$. We choose $0\leq \lambda\leq 1$
	such that there exist  two clopen sets such that 
	\begin{enumerate}
		\item[(C1')]
		$\beta_1'=\alpha_1$ on the first set,
		\item[(C2')]
		$\beta_1'=\beta_{k+1}'\geq \alpha_1$, on the second set.
	\end{enumerate}
	Such a choice is possible by the discussion above.
		Observe that the conditions in \eqref{finitecondition2} continue to hold for $a$ and 
	\[
	b'=\beta_1'(P_1+\cdots+P_k)+\beta_{k+1}'P_{k+1}+\sum_{i=k+2}^n \beta_iP_i.
	\]
They hold	for $l\leq k$  because $\beta_1'\geq \alpha_1$ and for $l\geq k+1$ because
	\[
	\beta_1'E(P_1+\cdots+P_k)+\beta_{k+1}'E(P_{k+1})=\beta_1E(P_1+\cdots+P_k)+\beta_{k+1}E(P_{k+1}).
	\]
	We have already shown how to deal  with  the set where $\beta_1'=\alpha_1$  using the induction hypothesis. It remains to consider the case when $k=n$, i.e., 
	$\beta_1=\cdots=\beta_n\geq \alpha_1$. 	But in this case we clearly have that $a\leq b$.  So, by Lemma \ref{basicsmaj} (i), $a\smaj b$.
\end{proof}

\begin{proposition}\label{finitevNsmaj}
	Suppose that $M$ is a  finite von Neumann algebra. Let  $a,b\in M_+$.
	If $a\smajtau b$, then $a\smaj b$.
\end{proposition}	
\begin{proof}
	We can reduce the proof to the case that $a$ and $b$ have finite spectrum. For suppose $\|a-a'\|<\epsilon$ and $\|b-b'\|<\epsilon$ for some $a'$ and $b'$ of finite spectrum and some $\epsilon>0$. Then, relying on Lemma \ref{basicsmaj}, we deduce that  $(a'-2\epsilon)_+\smaj (a-\epsilon)_+$ and 
	$(b-\epsilon)_+\smaj b'$. Hence    $(a'-2\epsilon)_+\smajtau b'$. Suppose  we have shown that $(a'-2\epsilon)_+\smaj b'$. Then, again using Lemma \ref{basicsmaj}, 
	we obtain that $(a-4\epsilon)_+\smaj (a'-3\epsilon)_+\smaj(b'-\epsilon)_+\smaj b$.
	Since $\epsilon>0$ can be arbitrarily small,  we arrive at $a\smaj b$, as desired.
	So let us assume that $a$ and $b$ have finite spectrum.

Express $a$ and $b$ in the form \eqref{formab} of Proposition \ref{theform}:
	\[
	a=\sum_{i=1}^n\alpha_iP_i,\quad
	b=\sum_{i=1}^n\beta_iP_i.
	\]
(We have conjugated $b$ by a unitary so that the projections in $a$ and $b$ are the same.) We can 
	take central cut-downs and reduce to the case that $c_{P_i}=1$ for all $i=1,\dots,n$.  From $a\smajtau b$ we deduce from  Lemma \ref{tracialineq1} that 
	\[
	\tau\Big(\sum_{i=1}^k \alpha_iP_i\Big)\leq \tau\Big(\sum_{i=1}^k \beta_iP_i\Big)
	\]
	for all $\tau\in \T(M)$ and all $k=1,\ldots,n$. Letting $\tau$ range through traces of the form $\delta_x\circ E$, where $\delta_x$ is a point evaluation on the center, 
	we deduce that  \eqref{finitecondition} from Proposition \ref{finitealg} holds. The desired result now follows from Proposition \ref{finitealg}.
\end{proof}

Recall that $E\colon M\to Z$ denotes the center valued trace of $M$ (whenever $M$ is assumed to be a finite von Neumann algebra).
 
\begin{proposition}\label{finitealgmaj}
Suppose that $M$ is a finite von Neumann algebra. Let $a,b\in M_+$ be positive contractions of the form \eqref{formab} in Proposition \ref{theform}.  
	Let $r\geq 0$.  If
	\begin{enumerate}
		\item [(a)] $\displaystyle \sum_{i=1}^k (\alpha_{i}-r)_+E(P_i)\leq \sum_{i=1}^k \beta_iE(Q_i)$ for all $k=1,\dots,n$,
	\end{enumerate}	
	and  
	\begin{enumerate}
		\item [(a')] $\displaystyle \sum_{i=k}^n (1-\alpha_{i}-r)_+E(P_i)\leq \sum_{i=k}^n (1-\beta_i)E(Q_i)$ for all $k=1,\dots,n$,
	\end{enumerate}	
	then there exists $b'\in M_+$ such that $b'\maj b$ and   $\|a-b'\|\leq r$.
\end{proposition}	
\begin{proof}
	Conjugating $b$ by a unitary	we may assume that $P_i=Q_i$ for all $i$.
Passing to central cut-downs $e_jM$, for suitable projections $e_1,\ldots,e_N\in Z$ that partition the unit, we may also assume that $c_{P_i}=1$ for all $i$.
		We will  proceed by induction on $n$ under the additional assumptions that $P_i=Q_i$ and $c_{P_i}=1$ for all $i$.

		If $n=1$ then
	$a=\alpha_1\cdot 1$ and $b=\beta_1 \cdot 1$ are multiples of the identity.
	From condition (a)  we  deduce that $(\alpha_1-r)_+\leq \beta_1$ whereas from (a') we deduce that $(1-\alpha_1-r)_+\leq 1-\beta_1$. Together they imply that $\|\alpha_1-\beta_1\|\leq r$.
	
	Let us assume now by induction that the proposition  is true when the number of projections $P_i$ is  less than a
	given $n$. Let $a$ and $b$ be as in the statement of the lemma. 
	From condition (a) with $k=1$ and from $c_{P_1}=1$ we deduce that $\beta_1\geq (\alpha_1-r)_+$.
	Just as we did before in the proof of Proposition \ref{finitealg}, let us apply  Lemma \ref{pinching}  in $\beta_1P_1+\beta_2P_2$ with a suitable  central element $0\leq \lambda\leq 1$
	(to be specified soon) so as to obtain $\rho\in Z_+$ and an element $\beta_1'P_1+\beta_2'P_2$ majorized by $\beta_1P_1+\beta_2P_2$. We have that 
	\[
	\beta_1\geq \beta_1'\geq \rho\geq \beta_2'\geq \beta_2,
	\]
	and that
	\[
	\beta_1E(P_1)+\beta_2E(P_2)=\beta_1'E(P_1)+\beta_2'E(P_2).
	\] 
	Let
	\[
	b'=\beta_1'P_1+\beta_2'P_2+\sum_{i=3}^n\beta_iP_i.
	\] 
Then for any 
	$\lambda\in Z$ such that $0\leq \lambda\leq 1$ the inequalities in (a), applied now to $a$ and $b'$,   hold  except possibly for $k=1$. The inequalities in (a') also  hold for $a$ and $b'$, except possibly for $k=2$. 
	Let us choose $\lambda$ such that  a each point of the spectrum of $Z$ either $\lambda=0$ or one of these two inequalities, $k=1$ in (a) or $k=2$ in (a'),  becomes an equality while the other one remains valid. More specifically, we choose a central element $0\leq \lambda\leq 1$ 
	such that the center is partitioned into three clopen sets satisfying the following conditions:
	\begin{enumerate}
		\item[(C1)]
		$\beta_1'=(\alpha_1-r)_+$, $\beta_1'\geq \beta_2'$, and  
		\[
		\sum_{i=2}^n (1-\alpha_{i}-r)_+E(P_i)\leq (1-\beta_2')E(P_2)+\sum_{i=3}^n (1-\beta_i)E(P_i)
		\] 
		on the first set,
		\item[(C2)]
		$\beta_1'\geq (\alpha_1-r)_+$, $\beta_1'\geq \beta_2'$, and  
		\begin{equation}\label{equaltrace}
			\sum_{i=2}^n (1-\alpha_{i}-r)_+E(P_i)=(1-\beta_2')E(P_2)+\sum_{i=3}^n (1-\beta_i)E(P_i)
		\end{equation} 
		on the second set,
		\item[(C3)]
		$\beta_1'\geq (\alpha_1-r)_+$, $\beta_1'=\beta_2'$, and  
		\[
		\sum_{i=2}^n (1-\alpha_{i}-r)_+E(P_i)\leq (1-\beta_2')E(P_2)+\sum_{i=3}^n (1-\beta_i)E(P_i)
		\] 
		on the third set.
	\end{enumerate}	
	To see that such a choice of $\lambda$  is possible, notice first that the inequalities 
	\[
	\beta_1'\geq (\alpha_1-r)_+
	\] 
	and 
	\[
	\sum_{i=2}^n (1-\alpha_{i}-r)_+E(P_i)\leq (1-\beta_2')E(P_2)+\sum_{i=3}^n (1-\beta_i)E(P_i),
	\]
	when put in terms of $\lambda$, take the general form 
	\[
	\kappa_1 \lambda\geq \gamma_1\hbox{ and }
	\kappa_2 \lambda\geq \gamma_2
	\]
	for some $\kappa_1,\kappa_2\in Z_+$ and $\gamma_1,\gamma_2\in Z_{sa}$ such that $\kappa_1\geq \gamma_1$ and $\kappa_2\geq \gamma_2$ (i.e., the inequalities are valid for $\lambda=1$). (In fact, $\kappa_1=\beta_1-\rho,$ $\gamma_1=(\alpha_1-r)_+-\rho$,  $\kappa_2=(\rho-\beta_2)E(P_2)$,
	and 
	\[
	\gamma_2=\sum_{i\geq 2}(1-\alpha_i-r)_+E(P_i)-\sum_{i\geq 2}(1-\beta_i)E(P_i)-(1-\rho)E(P_2).)
	\]	
	Let us choose
	\[
	\lambda=\max((\gamma_1)_+/\kappa_1,(\gamma_2)_+/\kappa_2).
	\] 
	These fractions are well defined in $Z$  because $Z$ is an abelian von Neumann algebra and $\kappa_1\geq (\gamma_1)_+$ and $\kappa_2\geq (\gamma_2)_+$.	Let us show that $\lambda$ is as desired. It is clear that $0\leq \lambda\leq 1$. Exploiting that
$\widehat Z$ is extremally disconnected, let us partition $\widehat Z$ into four clopen sets $X_1,X_2,X_3,X_4$ such that 
$\gamma_1\leq 0$ and $\gamma_2\leq 0$ on $X_{1}$,
$\gamma_1\geq 0$ and $\gamma_2\leq 0$ on $X_{2}$,
$\gamma_1\leq 0$ and $\gamma_2\geq 0$ on $X_{3}$, and
$\gamma_1\geq 0$ and $\gamma_2\geq 0$ on $X_{4}$.
It is straightforward to check that $\lambda=0$ on $X_1$. Thus, on this set we find ourselves in   case (C3) above. It can also be checked that   $\kappa_2\lambda=\gamma_2$ on $X_2$  and  $\kappa_1\lambda=\gamma_1$ on $X_3$. This values of $\lambda$ yield cases (C1) and (C2) above, respectively. 
Finally, partition $X_4$ into two clopen sets such that $\gamma_1\kappa_2\geq \gamma_2\kappa_1$ on one set and $\gamma_1\kappa_2\leq \gamma_2\kappa_1$ on the second. On the first of these sets we have that $\kappa_1\lambda=\gamma_1$ and on the other that $\kappa_2\lambda=\gamma_2$ (yielding again cases (C1) and (C2) above).

	Since $b'\maj b$, it suffices to prove that $a\maj b'$. Equivalently, it suffices  to prove the proposition with $b'$ in place of $b$. So  let us rename $\beta_1'$ and $\beta_2'$ as $\beta_1$ and $\beta_2$ and now assume that the conditions (C1)--(C3) for the three clopen sets described above hold for $\beta_1$ and $\beta_2$.
	
	Let us show that on the clopen sets  satisfying (C1) and (C2) we can argue by induction. 
	Indeed, restricting to the first set (while retaining the same names for our variables) we have that 
	\begin{align*}
	\sum_{i=2}^k (\alpha_{i}-r)_+E(P_i) &\leq \sum_{i=2}^k \beta_iE(Q_i)\hbox{ for all }k=2,\dots,n,\\
 \sum_{i=k}^n (1-\alpha_{i}-r)_+E(P_i) &\leq \sum_{i=k}^n (1-\beta_i)E(Q_i)\hbox{ for all }k=2,\dots,n.
 \end{align*}
Thus	
	\[
	a''=\sum_{i=2}^n\alpha_iP_i\hbox{ and }b''=\sum_{i=2}^n\beta_1P_i
	\] 
	satisfy the conditions (a) and (a') in the algebra $PMP$, where 
	$P=P_2+\cdots+P_n$. 
	(To see this we use that the center valued trace $E_P\colon PMP\to PZ$ can be computed to be $E_P(x)=\frac{E(x)}{E(P)}P$.)
	Hence,  by the 
	induction hypothesis applied in $PMP$, there exists
	$b'''\in PMP$  majorized by $b''$ and within $r$ distance of $a''$. The element $\beta_1P_1+b''$ is within $r$ of $a$ and $\beta_1P_1+b''\maj b$.
This proves  the induction step.
	
	Suppose now that we are in the second set. Consider the elements 
	\[
	a''=\sum_{i=2}^n(1- (1-\alpha_{i}-r)_+)P_i, \quad b''=\sum_{i=2}^n \beta_iP_i
	\]
	in $PMP$, where $P=P_2+\cdots+P_n$. From the conditions (a') applied to $a$ and $b$  we get that $a''$ and $b''$ satisfy  the conditions (a') with $r=0$. Moreover, from \eqref{equaltrace} we deduce that the center-valued traces
	of these two elements agree, i.e., $E_P(a'')=E_P(b'')$ (recall that we have relabeled $\beta_2'$ as $\beta_2$, so \eqref{equaltrace} is now valid with $\beta_2$ in place of $\beta_2'$). This in turn implies  that $a''$ and $b''$ satisfy the conditions (a) with $r=0$ as well. By the induction hypothesis with $r=0$  applied in the von Neumann algebra $PMP$  we get that
	$a''\maj b''$ in $PMP$.  Notice that 
	\[
	1- (1-\alpha_{i}-r)_+=\min(1,\alpha_i+r).
	\] 
	From this we easily deduce that $a''$ is within a distance $r$ of $\sum_{i=2}^n \alpha_iP_i$. Now, from the condition (a') applied to $a$ and $b$ with $k=1$, and keeping the equality \eqref{equaltrace} in mind, we deduce that $(\alpha_1+r)E(P_1)\geq \beta_1E(P_1)$.  This implies that $\alpha_1+r\geq \beta_1$ (since $c_{P_1}=1$, which implies that the subset 
	of $\widehat Z$ where  $E(P_1)$ is strictly positive is dense in $\widehat Z$). Similarly, 	from the condition (a) with $k=1$ we deduce  that  $\beta_1\geq \alpha_1-r$.
	So $\|\alpha_1-\beta_1\|\leq r$. Therefore, $\beta_1P_1+a''$ is within a distance $r$ of $a$ and  $\beta_1P_1+a''\maj b$. This again proves the induction step in this case.
	
	Let us examine now the third set, where $\beta_1=\beta_2$ while the conditions (a) and (a') remain valid.  Suppose more generally that 
	for some $k=2,\ldots,n$ we have that $\beta_1=\cdots=\beta_k$ while the conditions (a) and (a') are valid.  Suppose first that $k<n$. Let us apply Lemma \ref{pinching} to the element
	$\beta_1(P_1+\cdots +P_k)+\beta_{k+1}P_{k+1}$ with a suitable central element $0\leq \lambda\leq 1$ (to be specified soon). Call $\beta_1'(P_1+\cdots +P_k)+\beta_{k+1}'P_{k+1}$ the resulting element.
	As before, we can choose $\lambda$ such that the conditions (a) and (a')
	remain valid for $a$ and 
	\[
	b'=\beta_1'(P_1+\cdots+P_k)+\beta_{k+1}'P_{k+1}+\sum_{i>k+1} \beta_iP_i
	\]
	and such  that either one of the following three cases occurs after restricting to suitable clopen sets that partition $\widehat Z$: 
	\begin{enumerate}
		\item[(C1')]
		$\beta_1'=(\alpha_1-r)_+$, 
		\item[(C2')]	
		\begin{equation}\label{equaltrace2}
			\sum_{i=k+1}^n (1-\alpha_{i}-r)_+E(P_i)= (1-\beta_{k+1}')E(P_{k+1})+\sum_{i=k+2}^n (1-\beta_i)E(P_i),
		\end{equation}
		\item[(C3')]
		$\beta_1'=\beta_{k+1}'$.
	\end{enumerate} 
	Let us rename $\beta_1'$ and $\beta_{k+1}'$ as $\beta_1$ and $\beta_{k+1}$, respectively. We have already dealt with the first of these three cases. The second is dealt with similarly as before: The elements
	\[
	a''=\sum_{i=k+1}^n (1-(1-\alpha_{i}-r)_+)P_i\hbox{ and }b''=\sum_{i=k+1}^n\beta_iP_i 
	\]
	satisfy the induction hypotheses with $r=0$ in the von Neumann algebra $PMP$, where 
	$P=P_{k+1}+\cdots+P_n$. On the other hand, keeping in mind the equality \eqref{equaltrace2}, we deduce that 
	\[
	a'''=\sum_{i=1}^k\alpha_iP_i, \quad b'''=\sum_{i=1}^k \beta_iP_i
	\]
	satisfy conditions (a) and (a') with the same $r$ in the von Neumann algebra $(1-P)M(1-P)$.
	We can thus apply the induction hypothesis in both cases to get the desired result.
	
	The remaining case to be considered is when $k=n$, i.e.,  $\beta_1=\cdots=\beta_n$ and the conditions (a) and (a') are valid. From the condition (a) with $k=1$ we deduce that
	$\beta_1+r\geq \alpha_1$, while from the condition (a') with $k=n$ we deduce that $\alpha_n+r\geq \beta_n$ (here we use that $c_{P_i}=1$ for all $i$). This clearly implies that $\|\alpha_i-\beta_i\|\leq r$ for all $i$ implying that
	$\|a'-b'\|\leq r$, as desired.
\end{proof}

\begin{proposition}\label{finitevNmaj}
	Suppose that $M$ is a finite von Neumann algebra. Let $r\geq 0$. If $a,b\in M_+$ are contractions such that $(a-r)_+\smajtau b$
	and $(1-a-r)_+\smajtau 1-b$, then $a$ is within a distance $r$
	of  $\overline{\co\{ubu^*\mid u\in \U(M)\}}$.
\end{proposition}	

\begin{proof}
	Let $\epsilon>0$.	 Let $a'$ and $b'$ be positive contractions of finite spectrum   such that $\|a-a'\|<\epsilon/2$ and $\|b-b'\|<\epsilon/2$. Then,  using Lemma \ref{basicsmaj},  we find that  $(a'-r-\epsilon)_+\smajtau b'$ and 
	$(1-a'-r-\epsilon)\smajtau 1-b'$ (see the proof of Proposition \ref{finitevNsmaj}). Let us express $a'$ and $b'$ in the form of \eqref{formab} from Proposition \ref{theform}:
	\[
	a'=\sum_{i=1}^n \alpha_iP_i,\quad b'=\sum_{i=1}^n \beta_i Q_i.
	\]
	Conjugating $b'$ by a unitary  assume that $Q_i=P_i$ for all $i$. Cutting down the center by central projections  we assume that
	$c_{P_i}=1$ for all $i$.  By Lemma \ref{tracialineq1},  from $(a'-r-\epsilon)_+\smajtau b'$ we deduce  that
	\begin{enumerate}
		\item [(a)] $\displaystyle \sum_{i=1}^k (\alpha_{i}-r-\epsilon)_+E(P_i)\leq \sum_{i=1}^k \beta_iE(P_i)$ for all $k=1,\dots,n$,
	\end{enumerate}	
	and from $(1-r-\epsilon-a')_+\smajtau (1-b')$ that
	\begin{enumerate}
		\item [(a')] $\displaystyle \sum_{i=k}^n (1-\alpha_{i}-r-\epsilon)_+E(P_i)\leq \sum_{i=k}^n (1-\beta_i)E(P_i)$ for all $k=1,\dots,n$,
	\end{enumerate}	
	By Proposition \ref{finitealgmaj}, there exists $b''$ majorized by $b'$ and within $r+\epsilon$ distance of $a'$. Since $\epsilon$ can be arbitrarily small, this proves the proposition.
\end{proof}

We now proceed to extend Propositions \ref{finitevNsmaj} and \ref{finitevNmaj} to arbitrary von Neumann algebras. This is accomplished in Propositions \ref{vNsubmajorization}
and \ref{vNmajorization} below.

\begin{lemma}\label{tracialineq2}
	Let $a,b\in M_+$ be  as follows:
	\[
	a=\sum_{i=1}^n \alpha_iP_i,\quad b=\sum_{i=1}^n\beta_iP_i,
	\]
	where $(P_i)_{i=1}^n$ are orthogonal projections adding up to 1 and such that $c_{P_i}=1$ for all $i$ and where $(\alpha_i)_{i=1}^n$ and $(\beta_i)_{i=1}^n$ are decreasing nonnegative scalar  coefficients. Suppose that $a\smajtau b$.
	Then
	
	\begin{enumerate}[(a)]
		\item
		For all traces $\tau\in \T(M)$ and all $k=1,\dots,n$ we have		
		\[
		\tau\Big(\sum_{i=1}^k \alpha_iP_i\Big)\leq \tau\Big(\sum_{i=1}^k \beta_iP_i\Big).
		\]

		\item 
		$\alpha_1\leq \beta_1$. 
		
		\item 
		For each $k=2,\dots, n$ 	
		if $\alpha_k>\beta_k$  then $P_k\propto \sum_{i<k} P_i$.
		(Here $P\propto Q$ means that $P\smaj Q^{\oplus N}$ for some $N$.)
	\end{enumerate}
\end{lemma}

\begin{proof}
	Conditions (a) and (b) follow at once from Lemma \ref{tracialineq1}.
	Suppose now that $\alpha_k>\beta_k>0$  for some $k$. 
	By Lemma \ref{tracialineq1}, $\sum_{i=1}^k\alpha_iP_i\smajtau\sum_{i=1}^{k}\beta_iP_i$.
	Hence,
	\[
	\tau\Big(\sum_{i=1}^k (\alpha_i-\beta_k)P_j\Big)\leq \tau\Big(\sum_{i=1}^{k-1}(\beta_i-\beta_k) P_i\Big)
	\]
	for all $\tau\in \T(M)$.
	Since we have assumed that $\alpha_k-\beta_k>0$  this implies 
	that 
	\[
	\tau(P_k)\leq N\tau(P_1+\cdots+P_{k-1})
	\] for all $\tau\in \T(M)$ and some suitable positive integer $N$ (e.g., $N\geq \frac{\beta_i-\beta_k}{\alpha_k-\beta_k}$ for all $i$). By \cite[Theorem 8.4.3 (vii)]{KadRing},  this implies that 
	$P_k\propto P_1+\cdots + P_{k-1}$.	
\end{proof}

We start with the  submajorization result. First, a lemma. 
\begin{lemma}\label{pismaj}
	Let $P_1,\ldots,P_n$ be pairwise orthogonal projections such that $P_1$
	is properly infinite and $P_i\precsim P_1$ for all $i$. Let $\alpha_1,\ldots,\alpha_n$
	be  central positive elements such that $\alpha_i\leq \alpha_1$ for all $i$. Then 
	\begin{align*}
		\sum_{i=1}^n\alpha_iP_i\smaj \alpha_1P_1.
	\end{align*} 
\end{lemma}
\begin{proof}
	Let us write $P_1=P_1'+Q_2+\cdots+Q_n$, where $P_1',Q_2,\ldots,Q_n$ are pairwise orthogonal projections such that $P_1'\sim P_1$ and $Q_i\sim P_i$ for all $i\geq 2$.  Let $v\in M$ be a partial isometry such that 
	$vP_1'v^*=P_1$ and $vQ_iv^*=P_i$ for $i\geq 2$. Then
	\[
	v(\alpha_1P_1)v^*=\alpha_1vP_1'v^*+\sum_{i=2}^n\alpha_1vQ_iv^*=
	\sum_{i=1}^n\alpha_1P_i \geq \sum_{i=1}^n\alpha_iP_i.
	\] 
The result now follows from Lemma \ref{basicsmaj} (i).
\end{proof}

\begin{proposition}\label{vNsubmajorization}
	If $a,b \in M_+$ are such that  
 $a\smajtau b$, then $a\smaj b$.	
\end{proposition}	

\begin{proof}
	Arguing as in the proof of Proposition \ref{finitevNsmaj}  we can reduce the proof to the case	that $a$ and $b$
	have finite spectra. We then put them in the form \eqref{formab} from Proposition \ref{theform} assuming further that
	$P_i=Q_i$ for all $i$ (conjugating $b$ by a unitary if necessary):
	\begin{align*}
		a=\sum_{i=1}^n\alpha_iP_i,\quad
		b=\sum_{i=1}^n\beta_iP_i.
	\end{align*} 
Again arguments as in the proof of Proposition \ref{finitevNsmaj}  allow us to  assume that the coefficients $(\alpha_i)_{i=1}^n$ and $(\beta_i)_{i=1}^n$ have finite spectrum.

	Notice that if  $e$ is a central projection then the hypothesis of the theorem hold for 
	$ea$ and $eb$ in $eM$ (by Lemma \ref{centrallambda}).  
	On the other hand, if central projections $(e_j)_{j=1}^N$ partition the unit and we have proven the theorem for $e_ja$ and $e_jb$ in $e_jM$ for all $j$ then we conclude the same for $a$ and $b$. This allows us to make the following reductions:
	\begin{enumerate}[(1)]
		\item
		each $P_i$ is either finite or properly infinite for all $i$,
		\item
		the projections $P_i$ are pairwise orthogonal, pairwise Murray-von Neumann comparable, and add up to 1, 
		\item
		$c_{P_i}=1$ for all $i$.
	\end{enumerate}

	Recall that we have assumed that the central coefficients $(\alpha_i)_{i=1}^n$
	and $(\beta_i)_{i=1}^n$ have finite spectra. By passing to cut-downs of $M$ by central projections we can assume that
	these coefficients are scalars. Observe that the decreasing ordering of $(\alpha_i)_{i=1}^n$
	and $(\beta_i)_{i=1}^n$ is maintained by doing this and that properties (1)--(3) above are not destroyed in the process.  Thus,  we further assume that
	\begin{enumerate}
		\item[(4)]
		the coefficients $(\alpha_i)_{i=1}^n$
		and $(\beta_i)_{i=1}^n$ are decreasing scalars.
	\end{enumerate}

	We proceed by induction on the number of projections. If $n=1$ then Lemma \ref{tracialineq2} (b) implies that $\alpha_1\leq \beta_1$. Hence $a\leq b$.
	
	Let us consider the general case now.  
	The case when  all the projections $P_i$ are finite has already been dealt with in Proposition \ref{finitevNsmaj}.
	So let us  assume that  one of the projections is properly infinite. Let $P_k$ be a  projection larger than the rest in the Murray-von Neumann sense. By assumption, $P_k$ is properly infinite.

	\emph{Case $k<n$}.  
	From $a\smajtau b$ we know, by Lemma \ref{tracialineq1},  that
	$\sum_{i=1}^{k}\alpha_iP_i\smajtau\sum_{i=1}^{k}\beta_iP_i$. This relation also holds in the hereditary subalgebra $(P_1+\cdots+P_k)M(P_1+\cdots+P_k)$ (by Lemma \ref{hereditary}). 
By induction, 
	$\sum_{i=1}^k\alpha_iP_i\smaj\sum_{i=1}^k\beta_iP_i$. Hence
	\[
	\sum_{i=1}^k\alpha_iP_i+\sum_{i>k}\beta_iP_i\smaj b.
	\]
	But $\sum_{i=k}^n \alpha_i P_i\smaj \alpha_kP_k+\sum_{i>k}\beta_iP_i$,
	by Lemma \ref{pismaj}.  Hence,
	\[
	a= \sum_{i=1}^{n}\alpha_iP_i\smaj \sum_{i=1}^{k-1}\alpha_iP_i+\alpha_kP_k+\sum_{i>k}\beta_iP_i\smaj b.
	\]

	\emph{Case $k=n$.}
	Suppose that $P_n$ is the largest projection in the Murray-von Neumann sense (and it is properly infinite).
	If $\alpha_n>\beta_n$ then  from condition (c) of Lemma \ref{tracialineq2}    we get that $P_n\propto\bigoplus_{i<n}P_i$.  But we have assumed that the projections $P_i$
	are pairwise Murray-von Neumann comparable. So   $P_n\propto P_{k'}$ for some $k'<n$. The projection $P_{k'}$ is also properly infinite (since it cannot be finite). Hence,  $P_n\precsim P_{k'}$. 
	We are then in a case previously dealt with, since $P_{k'}$ is  properly infinite and  larger than the other projections. So let us assume that $\alpha_n\leq \beta_n$. 
	
	By Lemma \ref{tracialineq1} we have  $\sum_{i=1}^{n-1}\alpha_iP_i\smaj \sum_{i=1}^{n-1}\beta_iP_i$, which also holds in the hereditary subalgebra $(P_1+\cdots+P_{n-1})M(P_1+\cdots+P_{n-1})$ (by Lemma \ref{hereditary}). 
	Hence, by induction, 
	\[
	\sum_{i=1}^{n-1}\alpha_iP_i\smaj \sum_{i=1}^{n-1}\beta_iP_i.
	\]
	Since   $\alpha_n\leq \beta_n$ we get that $a\smaj b$, as desired.
\end{proof}

\begin{lemma}\label{Psmaj}
	Let $P$ be a properly infinite projection such that $P\sim 1$. Let $a,b\in (1-P)M(1-P)$ be positive contractions. 
	\begin{enumerate}[(i)]
		\item
		If $(1-P)-a\smaj (1-P)-b$ 
		then $\beta P+a\maj \beta P+b$ for any scalar 
		$\beta$ such that $a,b\leq \beta\leq 1$. 	
		\item
		If $a\smaj b$ then $a+\alpha P\maj b+\alpha P$
		for any scalar $\alpha\geq 0$ such that $a,b\geq \alpha(1-P)$. 
	\end{enumerate}	
	
\end{lemma}	
\begin{proof}
	(i) By Lemma \ref{basicsmaj},    
	\[
	((1-P)-a-t)_+\smaj ((1-P)-t-b)_+\] 
	for any $t\in [0,\infty)$. Choosing $t=1-\beta$ we obtain
	that $\beta(1-P)-a\smaj \beta(1-P)-b$ in $(1-P)M(1-P)$.   Since $1-P\precsim P$ and $P$ is properly infinite, we can
	find countably many orthogonal copies of $1-P$ in $PMP$. So $(1-P)M(1-P)\otimes \mathcal K$ embeds in $M$ mapping $(1-P)M(1-P)$ to itself. By Proposition \ref{smajstabilization},
	submajorization in a C*-algebra  is equivalent to majorization in the unitization of the stabilization
	of that C*-algebra. Hence 
	$\beta(1-P)-a\maj \beta(1-P)-b$ in $M$. So, 
	\[
	\beta-(\beta(1-P)-a)=\beta P+a\] 
	is majorized by $\beta P+b$ in $M$, as desired.
	
	(ii) By (i) applied to $a'=(1-P)-a$ and $b'=(1-P)-b$ with $\beta=1-\alpha$
	we get that 
	\[
	(1-\alpha)P+(1-P)-a\maj (1-\alpha)P+(1-P)-b.
	\] 
	Hence
	\[
	1-((1-\alpha)P+(1-P)-a)=a+\alpha P
	\] is majorized by $b+\alpha P$, as desired.
\end{proof}

\begin{proposition}\label{vNmajorization}
	Let $r\in [0,\infty)$.
	Let  $a,b\in M_+$ be contractions such that 
	$(a-r)_+\smajtau b$ and $(1-a-r)_+\smajtau 1-b$. Then $a$ 
	is within a distance $r$ of $\co\{ubu^*\mid u\in \U(M)\}$.
\end{proposition}

\begin{proof}
	Let $a,b\in M_+$ be as in the statement of the theorem. Let $\epsilon>0$.
	Let us find $a',b'\in M_+$, contractions with finite spectrum, such that 
	$\|a-a'\|<\epsilon/2$ and $\|b-b'\|<\epsilon/2$. Express them in the form
	\[
	a'=\sum_{i=1}^n \alpha_iP_i,\quad b'=\sum_{i=1}^n\beta_i Q_i
	\]
	as in \eqref{formab} of Proposition \ref{theform}. From 
	$(a-r)_+\smajtau b$ and $(1-a-r)_+\smajtau 1-b$ we deduce that  
	$(a'-r-\epsilon)_+\smajtau b'$ and 
	$(1-a'-r-\epsilon)_+\smajtau 1-b'$. Having proven the theorem for $a'$ and $b'$
	it is clear that, by letting $\epsilon\to 0$, we deduce the theorem for $a$ and $b$.
	So let us instead assume that $a$ and $b$, as in the statement of the theorem,  have finite spectra. Conjugating $b$ by a unitary, we may also assume that $P_i=Q_i$ for all $i$. We assume further that the central coefficients $(\alpha_i)_{i=1}^n$ and $(\beta_i)_{i=1}^n$ have finite spectra, which can be attained by a small enough approximation when moving from $a$, $b$ to $a'$, $b'$ respectively.
	
	Notice that if  $e$ is a central projection then the hypothesis of the theorem hold for 
	$ea$ and $eb$ in $eM$ (by Lemma \ref{centrallambda}).  
	On the other hand, if central projections $(e_j)_{j=1}^N$ partition the unit and we have proven the theorem for $e_ja$ and $e_jb$ in $e_jM$ for all $j$ then we conclude the same for $a$ and $b$. This allows us to make the following reductions:
	\begin{enumerate}[(1)]
		\item
		each projection $P_i$ is either finite or properly infinite for all $i$,
		\item
			the projections $P_i$ are pairwise orthogonal, pairwise Murray-von Neumann comparable, and add up to 1, 
		\item
		$c_{P_i}=1$ for all $i$.
	\end{enumerate}
	In the case that all the projections 
	$P_1,\ldots,P_n,$ are finite, the unit $1$  is 
	finite  and so the desired conclusion follows from Proposition \ref{finitevNmaj}. 
	Thus, we can make the following additional assumption:
	\begin{enumerate}
		\item[(4)]
		at least one of the projections $P_i$ is properly infinite.
	\end{enumerate}	  
	
	Recall that we have assumed that the central coefficients $(\alpha_i)_{i=1}^n$
	and $(\beta_i)_{i=1}^n$ have finite spectra. Passing to cut-downs of $M$ by central projections that partition the unit we can assume that
	these coefficients are scalars. Observe that the decreasing ordering of the coefficients $(\alpha_i)_{i=1}^n$
	and $(\beta_i)_{i=1}^n$ is maintained by doing this and that properties (1)--(4) above are not destroyed in the process.  Thus,  we further assume that
	\begin{enumerate}
		\item[(5)]
		the coefficients $(\alpha_i)_{i=1}^n$
		and $(\beta_i)_{i=1}^n$ are decreasing scalars.
	\end{enumerate}

	By Lemma \ref{tracialineq2},  $(a-r)_+\smajtau b$ implies the following conditions:
	\begin{enumerate}
		\item[(a)]
		$\tau(\sum_{i=1}^k(\alpha_i-r)_+P_i )\leq \tau(\sum_{i=1}^k\beta_iP_i)$
		for all $\tau\in \T(M)$ and $k=1,\dots,n$,
		\item[(b)]
		$\beta_1\geq \alpha_1-r$,
		\item[(c)]
		If for some $k\geq 2$ we have that 
		$\alpha_k -r> \beta_k$
		then $P_k\propto P_{k'}$ for some $k'<k$. That is, $P_k$ is Murray von Neumann smaller
		than finitely many copies of some $P_{k'}$ with $k'<k$. (Indeed, by Lemma \ref{tracialineq2}, $P_k\propto\bigoplus_{i<k}P_i$. But we have assumed that the projections $P_i$
		are pairwise Murray-von Neumann comparable. So $\bigoplus_{i<k}P_i\propto P_{k'}$ for some $k'<k$.)  
	\end{enumerate}
	Let us call the conditions stated above left-to-right conditions. One derives similar conditions
	from  $(1-a-r)_+\smajtau 1-b$. They take the form
	\begin{enumerate}
		\item[(a')]
		$\tau(\sum_{i=k}^n(1-\alpha_i-r)_+P_i)\leq \tau(\sum_{i=k}^n(1-\beta_i)P_i)$
		for all $\tau\in \T(M)$ and $k=1,\dots,n$,
		\item[(b')]
		$\beta_n\leq \alpha_n+r$,
		\item[(c')]
		If for some $k\leq n-1$  we have that $\alpha_k+r< \beta_k$
		then $P_k$ is Murray von Neumann smaller than finitely many copies of some $P_{k'}$  with $k'>k$.
	\end{enumerate}
	We'll call the above right-to-left conditions.

	Let $k=1,\ldots,n$ be the least index such that $P_k$ is  larger (in the Murray-von Neumann sense) than the other projections. By assumption $P_k$
	is also properly infinite. Notice that by conditions (b) and (c) we cannot have that
	$\alpha_k>\beta_k+r$. So either $|\alpha_k-\beta_k|\leq r$ or 
	$\beta_k>\alpha_k+r$. We consider these two cases next:

	\emph{Case $|\alpha_k-\beta_k|\leq r$}. Let us write $P_k=P_k'+P_k''$, where
	$P_k\sim P_k'\sim P_{k}''$. Consider the pair of elements 
	\[
	a'=\sum_{i=1}^{k-1}\max(\alpha_i-r,\beta_k)P_i+\beta_kP_k',\quad b'=\sum_{i=1}^{k-1} \beta_iP_i+\beta_kP_k'
	\]
	and the pair
	\[
	a''=\beta_kP_k''+\sum_{i=k+1}^n\min(\alpha_i+r,\beta_k)P_i, \quad b''=\beta_kP_k''+\sum_{i=k+1}^n\beta_iP_i.
	\]
	We claim that $a'\maj b'$ in $PMP$, where $P=P_1+\cdots+P_{k-1}+P_k'$,
	and that $a''\maj b''$ in $(1-P)M(1-P)$. Since $b=b'+b''$ and $\|a-(a'+a'')\|\leq r$, the desired result will follow from this claim.
	
	Let us prove that   $a'\maj b'$ in $PMP$. If $k=1$ this holds trivially, so assume that $k>1$. Let  $1\leq i_0\leq k-1$ be the largest index such that $\alpha_i-r\geq \beta_k$  and if there is no such index set $i_0=0$. From $(a-r)_+\smajtau b$ and Lemma \ref{tracialineq1} we get that
	\[
	\sum_{i=1}^{i_0} (\alpha_i-r)P_i\smajtau \sum_{i=1}^{i_0}\beta_iP_i.
	\]
	Furthermore, by Proposition \ref{vNsubmajorization}, the above relation is in fact a submajorization. On the other hand,
	\[
	\sum_{i=i_0+1}^{k-1}\beta_kP_i\leq \sum_{i=i_0+1}^{k-1}\beta_iP_i.
	\]
	Hence,   
	\[
	\sum_{i=1}^{k-1}\max(\alpha_i-r,\beta_k)P_i\smaj\sum_{i=1}^{k-1} \beta_iP_i.
	\]   
	That $a'\maj b'$ now follows from 
	Lemma \ref{Psmaj} (ii).
	
	The proof that $a''\maj b''$  in $(1-P)M(1-P)$ is entirely analogous (recall that we have written $P=P_1+\cdots+P_{k-1}+P_k'$):
	By Lemma \ref{Psmaj} (i), it suffices to check that  
	\[
	(1-P-P_k'')-\sum_{i=k+1}^n\min(\alpha_i+r,\beta_k)P_i
	\] is submajorized by  
	\[
	(1-P-P_k'')-\sum_{i=k+1}^n\beta_iP_i.
	\] 
	To check this, let $i_0\geq k+1$ be the largest index such that 
	$\alpha_i+r\geq \beta_k$. Then 
	\[
	\sum_{i=k+1}^{i_0} (1-\beta_k)P_i\leq \sum_{i=k+1}^{i_0} (1-\beta_i)P_i.
	\]
	On the other hand,  from $(1-a-r)_+\smajtau 1-b$ and Lemma \ref{tracialineq1} we get that
	\[
	\sum_{i=i_0+1}^n(1-\alpha_i-r)P_i\smajtau \sum_{i=i_0+1}^n(1-\beta_i)P_i.
	\]
	Moreover, by Proposition \ref{vNsubmajorization}, this relation is of submajorization. 
	Hence 
	\[
	\sum_{i=k+1}^n(1-\min(\alpha_i+r,\beta_k))P_i\smaj\sum_{i=k+1}^n(1-\beta_i)P_i,
	\] 
	as desired.

	\emph{Case $\beta_k> \alpha_k+r$}. 
	By condition (c'), there must exist an index $k'>k$ such that $P_{k'}$ is also properly infinite
	and larger than every other projection. Let $k'$ be the largest such index. 
	Notice that we cannot have that $\beta_{k'}>\alpha_{k'}+r$ by conditions (b') and (c') from the right-to-left conditions.
	So we must have that either $|\alpha_{k'}-\beta_{k'}|\leq r$
	or that $\alpha_{k'}> \beta_{k'}+r$. The first of these two cases has already been dealt with. So let us assume that  $\alpha_{k'}>\beta_{k'}+r$.
	
	We claim that $b$ majorizes
	\[
	a'=\sum_{i=1}^{k-1} (\alpha_i-r)_+P_i+\sum_{i=k}^{k'}\alpha_iP_i+\sum_{i=k'+1}^n(\alpha_i+r)P_i.
	\]
	Since $a'$ is within a distance $r$ of $a$, this is sufficient to complete  the proof of this case.
Let us prove our claim.	 Notice first that, as argued in the previous paragraphs, from Lemma \ref{Psmaj} (i) we obtain the majorization
	\begin{equation}\label{maj3}
		\beta_kP_k+\sum_{i=k'+1}^n (\alpha_i+r)P_i\maj \beta_kP_k+\sum_{i=k'+1}^n \beta_iP_i.
	\end{equation}
	in $(P_k+P_{k'+1}+\cdots+P_n)M(P_k+P_{k'+1}+\cdots+P_n)$. Similarly, from Lemma \ref{Psmaj} (ii) we obtain that
	\begin{equation}\label{maj4}
		\sum_{i=1}^{k-1}(\alpha_i-r)P_i+\beta_{k'}P_{k'}\maj \sum_{i=1}^{k-1}\beta_iP_i+\beta_{k'}P_{k'}.
	\end{equation}
	in $(P_1+\cdots+P_{k-1}+P_{k'})M(P_1+\cdots+P_{k-1}+P_{k'})$.
	We will be done once we have shown that
	\[
	\sum_{i=k}^{k'}\alpha_iP_i\maj \sum_{i=k}^{k'}\beta_iP_i,
	\]
	in $(P_k+\cdots+P_{k'})M(P_k+\cdots+P_{k'})$. Let us show this. We have 
	\[
	\sum_{i=k}^{k'-1}\alpha_iP_i\smaj \sum_{i=k}^{k'-1}\beta_iP_i,\]
	by Lemma \ref{pismaj}.
	So 
	\[
	\sum_{i=k}^{k'-1}\alpha_iP_i+\beta_{k'}P_{k'}\maj \sum_{i=k}^{k'}\beta_iP_i,
	\] by Lemma \ref{Psmaj} (ii).
	Repeating the same argument, symmetrically, 
	\[
	\sum_{i=k+1}^{k'}(1-\alpha_i)P_i\smaj \sum_{i=k+1}^{k'-1}(1-\alpha_i)P_i+(1-\beta_{k'})P_{k'},
	\] 
	by Lemma \ref{pismaj}. So 
	\[
	\sum_{i=k}^{k'}\alpha_iP_i\maj \sum_{i=k}^{k'-1}\alpha_iP_i+\beta_{k'}P_{k'},
	\]
	by Lemma \ref{Psmaj} (i).
\end{proof}

\section{Majorization and submajorization in C*-algebras}\label{proofofmain}

\begin{proposition}\label{hahn-banacharg}
	Let $A$ be a C*-algebra.  Let $a,b\in A$.
	\begin{enumerate}[(i)]
		\item
		The distance from  $a$ to $\co\{dbd^*\mid d\in A^{**},\, \|d\|\leq 1\}$ is equal to the distance from $a$ to $\co\{dbd^*\mid d\in A,\, \|d\|\leq 1\}$.
		\item
		Suppose that $A$ is unital. Then the distance from $a$ to 
		$\co\{ubu^*\mid u\in \U(A^{**})\}$ is equal to the distance from $a$
to 	$\co\{ubu^*\mid u\in \U(A)\}$.
	\end{enumerate}
\end{proposition}
\begin{proof}
	(i) It is clear that   the distance from $a$ to 
	$\co\{dbd^*\mid d\in A,\, \|d\|\leq 1\}$ is greater than or equal to the distance from $a$
	to 	$\co\{dbd^*\mid d\in A^{**},\, \|d\|\leq 1\}$. Denote the latter distance by $r$. Let $\epsilon>0$.  Suppose that 
	\[
	\Big\|a-\frac 1 n\sum_{i=1}^n d_ibd_i^*\Big\|<r+\epsilon
	\]	
	for some contractions $d_1,\ldots,d_n\in A^{**}$. For each $i=1,\ldots,n$ let us find a net of contractions $(d_{i,\lambda})_{\lambda}$ in $A$
	such that $d_{i,\lambda}\to d_i$ in the ultrastrong* topology.
	Such a net exists by Kaplansky's density theorem. Then the ultrastrong* closure of the set
	\[
	\Big\{a-\frac 1 n\sum_{i=1}^n d_{i,\lambda}bd_{i,\lambda}^*\mid \lambda\Big\}
	\]
	intersects the ball $B_{r+\epsilon}(0)$. By Hahn-Banach's theorem,  the convex hull of this  set also intersects that ball.  A convex combination of  elements of this set  again has the form $a-a'$ with $a'$ a convex combination of  elements of the form $dbd^*$ with $d\in A$  a contraction.
	
	(ii) It is clear that   the distance from $a$ to 
	$\co\{ubu^*\mid u\in \U(A)\}$ is greater than or equal to the distance from $a$
	to 	$\co\{ubu^*\mid u\in \U(A^{**})\}$. Denote the latter distance by $r$. Let $\epsilon>0$.  Suppose that
	\[
	\Big\|a-\frac 1 n\sum_{i=1}^n u_ibu_i^*\Big\|<r+\epsilon
	\]	
	for some unitaries 
	 $u_i\in A^{**}$.  By Kaplansky's density theorem for unitaries, there exist  nets of unitaries $(u_{i,\lambda})_{\lambda} $ in $A$ converging to $u_i$ in the  ultrastrong* topology. Then the ultrastrong* closure of the set
	 \[
	 \Big\{a-\frac 1 n\sum_{i=1}^n u_{i,\lambda}bu_{i,\lambda}^*\mid \lambda\Big\}
	 \]
	 intersects the ball $B_{r+\epsilon}(0)$. This implies that the convex hull of this  set also intersects that ball.  But a convex combination of  elements of this set  again has the form $a-a'$ with $a'$ a convex combination of  elements of the form $ubu^*$ with  $u\in \U(A)$.
\end{proof}

\begin{theorem}\label{mainsubmaj}
	Let $A$ be a  C*-algebra.  Let $a,b\in A_{\sa}$. The distance from $a$ to the set $\co\{ dbd^*\mid d\in A,\, \|d\|\leq 1\}$ is equal to the infimum $r\in [0,\infty)$ such that
\begin{align}
\tau((a-r-t)_+) &\leq \tau((b-t)_+)\hbox{ for all $t\in [0,\infty)$ and all $\tau\in \T(A)$,}\label{submajcond1}\\
\tau((-a-r-t)_+) &\leq \tau((-b-t)_+)\hbox{ for all $t\in [0,\infty)$ and all $\tau\in \T(A)$.}\label{submajcond2}
\end{align}		
Moreover, if $r$ is such infimum then $(a-r)_+-(a+r)_-\smaj b$.
\end{theorem}	
\begin{proof}
Let $\tilde r\in (0,\infty)$ be such that  $\|a-b'\|<\tilde r$ for some  $b'\smaj b$. Then $(a-\tilde r)_+\smaj b'_+$ by Lemma \ref{basicsmaj} (ii). Also, $b'\smaj b$ implies that $(b')_+\smaj b_+$,   by Proposition \ref{reducetopositive}. Hence 
$(a-\tilde r)_+\smaj b_+$. Starting from $\|(-a)-(-b')\|<\tilde r$ and following the same line of reasoning we obtain that $(a+\tilde r)_-\smaj b_-$. Since submajorization implies tracial submajorization (Proposition \ref{easyimplication}), 
$(a-\tilde r)_+\smajtau b_+$ and $(a+\tilde r)_-\smajtau b_-$. These relations translate at once into \eqref{submajcond1} and \eqref{submajcond2} (for the number $\tilde r$). 
	
	Assume now that \eqref{submajcond1}--\eqref{submajcond2} hold for some 
	$r\in [0,\infty)$.  
	Let us show that $(a-r)_+ - (a+r)_-\smaj b$. Since the distance from $a$ to $(a-r)_+-(a+r)_-$ is $r$,
	this will complete the proof of the theorem.
	As remarked above, \eqref{submajcond1}--\eqref{submajcond2}
 can be restated as saying that $(a-r)_+\smajtau b_+$ and $(a+ r)_-\smajtau b_-$.   In view of Proposition \ref{reducetopositive}, it remains to show that $(a-r)_+\smaj b_+$ and $(a+r)_-\smaj b_-$. This boils down to showing    if $c,d\in A_+$ are such that   $c \smajtau d$ then  $c\smaj d$.  Let us prove this. It is  clear from $c\smajtau d$ in $A$ that $c\smajtau d$ in the von Neumann algebra $A^{**}$ (indeed, in any C*-algebra containing $A$), since traces 
	in $\T(A^{**})$ restrict to traces in $\T(A)$. Then, by Proposition \ref{vNsubmajorization}, 
	$c\smaj d$ in $A^{**}$.  Finally, by Proposition \ref{hahn-banacharg} (i), 
	$c\smaj d$ in $A$, as desired.
	\end{proof}	

\begin{theorem}\label{thmdistance}
	Let $A$ be a unital C*-algebra. Let $a,b\in A$ be selfadjoint elements. Then 
	the distance from $a$ to $\co\{ubu^*\mid u\in \U(A)\}$ is equal to the infimum $r\in [0,\infty)$
	such that
		\begin{align}
		 \tau((a-r-t)_+) &\leq \tau((b-t)_+)\hbox{ for all $t\in \R$ and all $\tau\in \T(A)$,}\label{smajconds1}\\
		 \tau((-a-r-t)_+ &\leq \tau((-b-t)_+)\hbox{ for all $t\in \R$ and all $\tau\in \T(A)$.}\label{smajconds2}
		 \end{align}
\end{theorem}	
\begin{proof}
If we replace $a$  by $a+s\cdot 1 $ and $b$  by
$b+s\cdot 1$ for some $s\in \R$ then neither the infimum $r$ satisfying  
\eqref{smajconds1}--\eqref{smajconds2}
nor the distance from $a$ to $\co\{ubu^*\mid u\in \U(A)\}$ is changed.   Thus, by choosing a sufficiently large $s$ we may assume that $a$ and $b$ are positive.	
A simple calculation also shows  that if we replace $a$ by $a/s'$ and $b$ by $b/s'$ for some $s'\in (0,\infty)$ then both the infimum $r$ satisfying \eqref{smajconds1}--\eqref{smajconds2} 
and the distance from $a$ to $\co\{ubu^*\mid u\in \U(A)\}$ get multiplied by a factor
of $1/s'$. Thus, by choosing a sufficiently large $s'$ we may assume that $a$
and $b$ are positive contractions. We do so henceforth.

Let $r\in (0,\infty)$  be any number satisfying 	 \eqref{smajconds1}--\eqref{smajconds2}. 
 From \eqref{smajconds1} we deduce that $
(a-r)_+\smajtau b$ while from \eqref{smajconds2} we deduce that  
$(1-a-r)_+\smajtau 1-b$.
Thus, by Proposition \ref{vNmajorization},  $a$ is within a distance $r$ of $\co\{ubu^*\mid u\in \U(A^{**})\}$. Then, by Proposition \ref{hahn-banacharg} (ii),
$a$ is within a distance $r$ of $\co\{ubu^*\mid u\in \U(A)\}$. This proves one inequality.

Let  $\tilde r\in (0,\infty)$ be any number such that $\|a-b'\|<\tilde r$ for some 
 $b'\in \co\{ubu^*\mid u\in \U(A)\}$. By Lemma \ref{basicsmaj} (ii), $(a-\tilde r)_+\smaj b'\smaj b$. Since submajorization implies tracial submajorization  (for positive elements) we have that $(a-\tilde r)_+\smajtau b$. That is,
 \[
 \tau((a-\tilde r-t)_+)\leq \tau((b-t)_+)
 \]
 for all $t\in [0,\infty)$ and all $\tau\in \T(A)$. Using that $b$ is positive we can  extend this inequality to  all $t<0$. Indeed, if $t<0$ then
 $(a-\tilde r-t)_+\leq (a-\tilde r)_+-t$, and so  
 \begin{align*}
 \tau((a-\tilde r-t)_+) &\leq \tau((a-\tilde r)_+)+\tau(-t)\\
 &\leq \tau(b)+\tau(-t)\\
 &=\tau((b-t)_+),
 \end{align*}
 for all $\tau\in \T(A)$.
 Thus, \eqref{smajconds1} holds for $\tilde r$.
Applying the same arguments starting from  $\|(1-a)-(1-b')\|<\tilde r$ we deduce that $\tau((1-a-\tilde r-t)_+\leq \tau ((1-b-t)_+)$ for all $\tau\in \T(A)$ and all $t\in \R$.
This is equivalent to \eqref{smajconds2}.
\end{proof}	

\begin{proof}[Proof of Theorem \ref{mainthm}]
	This is the case $r=0$ of Theorem \ref{thmdistance}.
\end{proof}

\begin{remark}\label{positivecontractions}
The following observation, whose verification is left to the reader, will be useful below: if $a,b\in A$ are positive contractions
the tracial inequalities in Theorem \ref{mainthm} (ii) are equivalent to the tracial submajorizations $a\smajtau b$ and $1-a\smajtau 1-b$.	
\end{remark}

Let us  explore some consequences of Theorem \ref{mainthm}.

\begin{corollary}\label{simple}
Let $A$ be a simple unital C*-algebra. Let $a,b\in A_{\sa}$.
\begin{enumerate}[(i)]
\item
If $A$ has  at least one non-zero bounded trace then $a\maj b$ if and only if
$\tau(a)=\tau(b)$ and $\tau((a-t)_+)\leq \tau((b-t)_+)$ for all $t\in \R$ and all bounded traces $\tau$ on $A$.
\item
If $A$ has no bounded traces then $a\maj b$ if and only if $\mathrm{sp}(a)\subseteq \co(\mathrm{sp}(b))$.
\end{enumerate}
\end{corollary}
\begin{proof}
The implications starting with $a\maj b$ in both (i) and (ii) are
straightforward from Theorem \ref{mainthm}. To prove the converse  
we will show in both cases that the tracial inequalities from Theorem \ref{mainthm} hold. 

(i) Let us suppose that $A$ has at least one non-zero bounded trace. 
Since $A$ is simple and unital, $\T(A)$ consists of the bounded traces on $A$ and the trace $\tau_\infty(a):=\infty$
for all $a\neq 0$ and $\tau_\infty(0):=0$. We have assumed that $\tau((a-t)_+)\leq \tau((b-t)_+)$ for all bounded traces and all $t\in \R$.  Let us  show that 
\begin{equation}\label{tauinftycheck}
\tau_\infty((a-t)_+)\leq \tau_\infty((b-t)_+)
\end{equation}
for all $t\in \R$. It suffices to show that the left side is zero for all $t\geq \|b_+\|$.
We have that 
\[
\tau((a-\|b_+\|)_+)\leq \tau((b-\|b_+\|)_+)=0,
\] 
for all bounded traces $\tau$. Since $A$ has at least one non-zero bounded trace---which is necessarily faithful because $A$ is simple---we get that $(a-\|b_+\|)_+=0$. This implies \eqref{tauinftycheck}.

Let us prove that $\tau((-a-t)_+)\leq \tau((-b-t)_+)$ for all $\tau\in \T(A)$ and all $t\in \R$.  Let $t\in \R$. Let $\tau$ be a bounded trace (which we assume defined on all $A$). Observe that 
\[
(-c-t)_+=(c+t)_+-(c+t),
\]
for any selfadjoint element $c$. Thus, as $\tau(a)=\tau(b)$,
\begin{align*}
\tau((-a-t)_+)&=\tau((a+t)_+)-\tau(a+t)\\
&\leq 
\tau((b+t)_+)-\tau(b+t)=\tau((-b-t)_+).
\end{align*}
 To get that $\tau_\infty((-a-t)_+)\leq \tau_\infty((-b-t)_+)$ we proceed as in the previous paragraph. Exploiting the existence of a non-zero (faithful) bounded trace
 we deduce that $(a+\|b_-\|)_-=0$ (since $(b+\|b_-\|)_-=0$), from which the desired inequality readily follows.

(ii) Suppose that $A$ has no non-zero bounded traces. Then $\T(A)$ consists only of $\tau_\infty$ and the zero trace. Since $\mathrm{sp}(a)\subseteq \co(\mathrm{sb}(b))$, we have that
  $\|a_+\|\leq \|b_+\|$ and $\|a_-\|\leq \|b_-\|$. It is readily verified from this that $\tau_\infty((a-t)_+)\leq \tau_\infty((b-t)_+)$
and $\tau_\infty((-a-t)_+)\leq \tau_\infty((-b-t)_+)$ for all $t\in \R$, as desired.
\end{proof}

\begin{theorem}
Let $A$ be a unital C*-algebra. Let  $a$  be a selfadjoint element in $A$. Then $0\in \overline{\co\{uau^*\mid u\in \U(A)\}}$
if and only if
\begin{enumerate}[(a)]
\item
$\tau(a)=0$ for all bounded traces $\tau$ on $A$, and
\item
in no nonzero quotient of $A$ can the image of $a$ be either invertible and positive or invertible and negative.
\end{enumerate}
\end{theorem}
\begin{proof}
The necessity of the conditions is relatively straightforward.  Since all the elements in the set  
$\overline{\co\{uau^*\mid u\in \U(A)\}}$ agree on bounded traces, we have (a).
If $a\geq \alpha\cdot  1$ for some $\alpha\in (0,\infty)$ then the same holds for all elements
in  $\overline{\co\{uau^*\mid u\in \U(A)\}}$, which prevents 0 from belonging to this set. Similarly, we cannot have that $a\leq -\alpha 1$. Moreover, if 0 is in the closure of the convex hull of the unitary conjugates of $a$ the same holds for the image of $a$ on any quotient. So we have (b).

Suppose now that (a) and (b) hold.  To prove the theorem we use Theorem \ref{mainthm}. We must check that $\tau((0-t)_+)\leq \tau((a-t)_+)$
for all $t\in \R$ and all $\tau\in \T(A)$. This boils down to showing that  
$\tau(t)\leq \tau((a+t)_+)$ for all $t>0$ and all $\tau\in \T(A)$. Let $t>0$. Suppose first that 
$\tau$ is a bounded trace (so assume that it is defined on all $A$). 
Evaluating $\tau$ on $(a+t)_+\geq a+t$ we get $\tau((a+t)_+)\geq \tau(t)$, as desired.
Suppose now that $\tau$ is unbounded.  Since  $\tau(t)=\infty$ we must show  that $\tau((a+t)_+)=\infty$. Equivalently, we must show  that $(a+t)_+$ is full, i.e., it
generates $A$ as a closed two-sided ideal.  But if this were not  the case   then in the quotient by the closed two-sided ideal
generated by $(a+t)_+$ we would have that $\overline a+t\leq 0$ (where $\overline a$ denotes the image of $a$ in this quotient). This contradicts (2). Thus, $(a+t)_+$
is full, as desired. Since $-a$ satisfies (1) and (2) too, we also arrive at $\tau((-0-t)_+)\leq \tau((-a-t)_+)$ for all $t\in \R$ and all $\tau\in \T(A)$.
 By Theorem \ref{mainthm}, $0\maj a$, as desired.
 \end{proof}

\section{Uniform majorization}\label{uniform}
 In this section we discuss the majorization relation in the context of regularity properties of C*-algebras. We show that  one has a uniform version of majorization holding across all C*-algebras of either one of the following classes: 
 \begin{enumerate}[(1)]
 	\item
 	C*-algebras  satisfying Blackadar's strict comparison of positive elements by traces, 
 	\item
 	 C*-algebras having a uniform bound on their  nuclear dimension.
 \end{enumerate}	 
 	  In both cases we derive the uniform  majorization
from the preservation of the relation of tracial submajorization  under products of C*-algebras in the given class (Propositions \ref{prodstrict} and \ref{prodnuc}).

Let us recall some  definitions. 
Let $A$ be a C*-algebra. Let $\mathcal K$ denote the C*-algebra of compact operators on a separable infinite dimensional Hilbert space. Let $\tau\in \T(A)$. We can extend $\tau$  to a trace on $(A\otimes\mathcal K)_+$
by setting 
\[
\tau((a_{i,j}))=\sum_{i=1}^\infty \tau(a_{i,i})
\] for all $(a_{i,j})_{i,j}\in (A\otimes \mathcal K)_+$. From $\tau$ we obtain
a ``dimension function" $d_\tau\colon (A\otimes \mathcal K)_+\to [0,\infty]$  defined as
\[
d_\tau(a)=\lim_n \tau(a^{\frac 1 n}),
\] 
for all $a\in (A\otimes \mathcal K)_+$.
(Alternatively, $d_\tau(a)$ is the norm of the restriction of $\tau$ to 
$\overline{a(A\otimes \mathcal K)a}$.) 

Next, let us recall the definition of the Cuntz comparison  relation among positive elements:
Given  positive elements $a,b\in A\otimes \mathcal K$, $a$ is said to be Cuntz subequivalent  to $b$ if there exist
 $e_1,e_2,\ldots\in A\otimes \mathcal K$ such that $e_nbe_n^*\to a$.
We denote this relation by $a\precsim_{\mathrm{Cu}}b$.

 The C*-algebra $A$ is said to have the property of strict comparison of positive elements by traces if
for all $a,b\in (A\otimes\mathcal K)_+$ and $\epsilon>0$   we have that
\begin{equation}\label{strictcomp}
d_\tau(a)\leq (1-\epsilon)d_\tau(b)\hbox{ for all $\tau\in \T(A)$ implies that  }
a\precsim_{\mathrm{Cu}}b.
\end{equation}
(Note: A number of different variations on ``strict comparison''  exist in the literature; e.g.,   one may  restrict $\tau$ to 
be a  bounded trace, or allow it to be a 2-quasitrace; one may restrict $a,b$ to be in $A$, etc.) 

We will make use of the topology on $\T(A)$ introduced in \cite{elliott-robert-santiago}. Let us recall it here: a net $(\tau_\lambda)_\lambda$ in $\T(A)$ converges to $\tau$
if for all $a\in A_+$ and $\epsilon>0$ we have
\[
\limsup_\lambda \tau_\lambda((a-\epsilon)_+)\leq \tau(a)\leq \liminf_\lambda\tau_\lambda(a).
\]
It is shown in \cite[Theorem 3.7]{elliott-robert-santiago} that $\T(A)$ is compact and Hausdorff under this topology.

The following variation on the strict comparison property has been introduced in \cite{ng-robert}: Let $K\subseteq \T(A)$ be a compact set.  Then $A$ is said to have strict  comparison of positive elements by traces in $K$ if for all $a,b\in (A\otimes \mathcal K)_+$ and $\epsilon>0$ it suffices  to let $\tau$ range through $K$ in \eqref{strictcomp} for this implication to hold.

The following proposition is essentially obtained in \cite{ng-robert}:

\begin{proposition}\label{prodstrict}
Let $A_1,A_2,\dots$ be C*-algebras with strict comparison of positive elements by traces. Let 
$a=(a_n)_{n=1}^\infty$ and $b=(b_n)_{n=1}^\infty$ be positive elements in 
$\prod_{n=1}^\infty A_n$
such that  $a_n\smajtau b_n$ for all $n$. Then $a\smajtau b$ in $\prod_{n=1}^\infty A_n$.
\end{proposition}

\begin{proof}
Let us regard $\T(A_n)$ embedded in  $\T(\prod_{n=1}^\infty A_n)$ via the map induced by the projection 
from $\prod_{n=1}^\infty A_n$ onto $A_n$. Let $K=\overline{\bigcup_{n=1}^\infty \T(A_n)}\subseteq \T(\prod_{n=1}^\infty A_n)$.
In the course of the proof of \cite[Theorem 4.1]{ng-robert} it is shown that the C*-algebra $\prod_{n=1}^\infty A_n$ has strict
comparison of positive elements by traces in $K$.
 The elements $a$ and $b$
from  the statement of the theorem satisfy that $\tau((a-t)_+)\leq \tau((b-t)_+)$ for all $\tau\in \bigcup_{n=1}^\infty \T(A_n)$ and $t\geq 0$ (this holds by assumption). Let us shows that  these inequalities extend to all traces in $K$. Let $\tau\in K$ and choose a net $\tau_\lambda\to \tau$ with $\tau_\lambda\in \bigcup_{n=1}^\infty \T(A_n)$. From  the definition of the topology in $\T(A)$   we get that
\begin{align*}
\tau((a-t-\epsilon)_+) &\leq \liminf\tau_\lambda((a-t-\epsilon)_+)\\
&\leq \liminf\tau_\lambda((b-t-\epsilon)_+)\\
&\leq \tau((b-t)_+),
\end{align*}
for all $t\geq 0$ and $\epsilon>0$.
Thus, $\tau((a-t-\epsilon)_+)\leq \tau((b-t)_+)$. Letting $\epsilon\to 0$ and using the lower semicontinuity of $\tau$ we get that
 $\tau((a-t)_+)\leq \tau((b-t)_+)$ for all $\tau\in K$ and all $t\geq 0$. Now,   \cite[Lemma 3.4]{ng-robert}   asserts that if a C*-algebra $A$ has strict comparison by traces in a compact set $K$ then for any given $c,d\in A_+$ if
$\tau(c)\leq \tau(d)$ for all $\tau\in K$ then $\tau(c)\leq \tau(d)$ for all $\tau\in \T(A)$. Applied in $A=\prod_{n=1}^\infty A_n$ with  $K$ as above  this lemma implies that $a\smajtau b$, as desired.
\end{proof}

\begin{proof}[Proof of Theorem \ref{uniformthm}]
Let $\epsilon>0$.  Suppose for the sake of contradiction that no $N$
as in the statement of the theorem exists.  Then there exist unital C*-algebras $A_1,A_2,\dots$ with
strict comparison by traces  and selfadjoint contractions  $a_n,b_n\in A_n$ such that $a_n\maj b_n$ for all $n$ but 
\[
\|a_n-\frac 1 n\sum_{i=1}^n u_ib_nu_i^*\|\geq \epsilon
\]
for all $n$-tuples of unitaries $u_1,\dots,u_n\in A_n$.  Let $a_n'=\frac {a_n+1}{2}$
and $b_n'=\frac{b_n+1}{2}$. Observe that these are positive contractions such that $a_n'\maj b_n'$ for all $n$ and 
\[
\|a_n'-\frac 1 n\sum_{i=1}^n u_ib_n'u_i^*\|\geq \frac \epsilon 2,
\]
for all $n$-tuples of unitaries $u_1,\dots,u_n\in A_n$. 
Consider the positive elements
$a=(a_n')_{n=1}^\infty$ and $b=(b_n')_{n=1}^\infty$ in $\prod_{n=1}^\infty A_n$.  Since $a_n'\maj b_n'$ for all $n$ we have,  by Proposition \ref{prodstrict}, that $a\smajtau b$.  
Also,   $1-a_n'\smajtau 1-b_n'$ for all $n$ and so $1-a\smajtau 1-b$.  By Theorem \ref{mainthm} (keeping Remark \ref{positivecontractions} in mind), we have  that
  $a\maj b$. Hence, there exists
$N\in \N$ and unitaries $w_1,w_2,\dots,w_N\in  \prod_{n=1}^\infty A_n$
such that 
\[
\|a-\frac 1 N\sum_{i=1}^N w_ibw_i^*\|<\frac{\epsilon}{2}.
\]  
Projecting onto $A_N$ we arrive at a contradiction.
\end{proof}

\begin{theorem}\label{uniformsub}
For each $\epsilon>0$ there exists $N\in \N$ such that if $A$ is a C*-algebra with strict comparison of positive elements by traces and $a,b\in A_{\sa}$ are contractions such that $a\smaj b$ then  
\[
\|a-\frac 1 N\sum_{i=1}^N d_ibd_i^*\|<\epsilon
\] 
for some contractions $d_1,\ldots,d_N\in A$.
\end{theorem}
\begin{proof}
It is easy to argue, using Proposition \ref{reducetopositive},  that it suffices 	to prove the theorem letting $a$
and $b$ range through all positive contractions. One can then proceed as in the proof
of Theorem \ref{uniformthm}, arguing by contradiction and relying on Proposition \ref{prodstrict}. The details are left to the reader.
\end{proof}	

Next we prove the same uniform majorization  among  C*-algebras with a uniform bound in their nuclear dimension. We start by recalling the definition of nuclear dimension and some background facts.

A completely positive contractive (c.p.c.) map $\phi\colon A\to B$  is  called of order zero  if it preserves orthogonality, i.e., $ab=0$ implies $\phi(a)\phi(b)=0$ for all $a,b\in A$. By \cite[Theorem 2.3]{winter-zacharias0}, such a map  has the form $\phi(a)=h\pi_\phi(a)$  where $\pi_\phi\colon A\to \mathrm M(C^*(\phi(A))$ is a homomorphism  and where $h\in \mathrm M(C^*(\phi(A)))$  is a positive element commuting with $\pi_\phi(A)$.
(Here $\mathrm M(C^*(\phi(A)))$ denotes the multiplier algebra of the C*-algebra generated by $\phi(A)$.) With the aid of this theorem one can easily deduce the preservation of various relations under c.p.c. order zero maps. For example, if $a=x^*x$ and $b=xx^*$ for some $x\in A$ then $\phi(a)=y^*y$ and $\phi(b)=yy^*$ for some $y\in B$ (we can choose $y=h^{1/2}\pi_\phi(x)$). The submajorization  relation is also preserved under c.p.c. order zero maps.  For if
$a,b\in A_{\sa}$ are such that  $a=\frac 1 N\sum_{i=1}^N d_ibd_i^*$ for some contractions $d_i\in A$, then 
$\phi(a)=\frac 1 N\sum_{i=1}^N \pi_\phi(d_i)\phi(b)\pi_\phi(d_i)^*$. Although
the contractions $\pi_{\phi}(d_i)$ belong to $M(C^*(\phi(A)))$ rather than $B$, by 
Lemma \ref{hereditary} (i) we still have that $\phi(a)\smaj \phi(b)$ in $B$. 
If, more generally, $a\smaj b$ in $A$, then an argument passing to limits readily proves that $\phi(a)\smaj \phi(b)$ in $B$.

Let $m\in \N$. Following Winter and Zacharias \cite{winter-zacharias} we say that a $\mathrm{C}^*$-algebra $A$ has  nuclear dimension  at most $m$
if for each finite set $F\subset A$ and $\epsilon>0$ there exist c.p.c.\ maps 
$A\stackrel{\psi_k}{\longrightarrow}C_k \stackrel{\phi_k}{\longrightarrow}A$
with $k=0,1,\dots,m$,
such that $C_k$ is a finite dimensional C*-algebra for all $k$, $\phi_k$ is an order zero map for all $k$, and 
\[
\|a-\sum_{k=0}^m\phi_k\psi_k(a)\|<\epsilon  \quad\hbox{for all }a\in F.
\]
In \cite[Proposition 3.2]{winter-zacharias}, Winter and Zacharias show that it is possible to arrange for the maps $\psi_k$ to be asymptotically of order zero. In this way one obtains c.p.c. order zero maps
\[
A\stackrel{\overline\psi_k}{\longrightarrow}N_k \stackrel{\overline\phi_k}{\longrightarrow}A_\infty
\]
for $k=0,\dots,m$, 
such that 
\[
\iota=\sum_{k=0}^m \overline\phi_k\overline\psi_k.
\] 
Here $A_\infty=(\prod_{\lambda} A_\lambda)/(\bigoplus_{\lambda} A_\lambda)$ is a sequence algebra over some upward directed set $\Lambda$,
$\iota\colon A\to A_\infty$
denotes the canonical embedding of $A$ in $A_\infty$ as ``constant sequences", and  
$N_k=(\prod_{\lambda} C_{k,\lambda})/(\bigoplus_{\lambda} C_{k,\lambda})$, where $C_{k,\lambda}$ is a finite dimensional C*-algebra  for all $\lambda\in \Lambda$ and all $k=0,\dots,m$.
 
\begin{lemma}\label{Nkhasstrict}
	Each C*-algebra $N_k$ as defined above has the property of strict comparison of positive elements by traces.
\end{lemma}	
\begin{proof}
	This is a consequence of $N_k$ being the quotient of a product of finite dimensional C*-algebras. More specifically, as remarked in the proof of Proposition \ref{prodstrict}, a product of C*-algebras with strict comparison by traces again has strict comparison by traces (in fact, by traces ranging in a suitable compact set $K$).  Since each $C_{k,\lambda}$ is finite dimensional it has strict comparison by traces. Thus, the same holds  for $\prod_\lambda C_{k,\lambda}$. Also,  the property of strict comparison by traces also passes to quotients. Indeed, by \cite[Proposition 3.6 (i)]{ng-robert}, strict comparison by traces is equivalent to ``strict comparison by 2-quasitraces and 2-quasitraces are traces". It is clear that if all the lower semicontinuous 2-quasitraces of a C*-algebra are traces the same holds for its quotients. Strict comparison by 2-quasitraces also  passes to quotients since, by \cite[Proposition 6.2]{elliott-robert-santiago}, it is equivalent to almost unperforation in the Cuntz semigroup   and the latter passes to quotients by \cite[Proposition 2.2]{robert-tikuisis} (it is called $0$-comparison in this reference). 
	\end{proof}	

\begin{lemma}
	For each $\epsilon>0$ there exists $N\in \N$ such that if $A$ is a C*-algebra of nuclear dimension at most $m$
	and $a,b\in A_+$ are such that $a\smajtau b$ then 
	\[
	(a-\epsilon)_+=\sum_{i=1}^{N(m+1)}x_i^*x_i\hbox{ and }\sum_{i=1}^{N(m+1)} x_ix_i^*\leq b,
	\]	
	for some $x_1,\ldots,x_{N(m+1)}\in A$.
\end{lemma}	
\begin{proof}
Let $A\stackrel{\overline\psi_k}{\longrightarrow}N_k \stackrel{\phi_k}{\longrightarrow}A_\infty$, for $k=0,\dots,m$,  be c.p.c. order zero maps as in the discussion above. Fix $k=0,\dots,m$. We have remarked above that c.p.c. order zero maps preserve the submajorization relation (which, by Theorem \ref{mainsubmaj}, agrees with $\smajtau$). Hence,  $a\smajtau b$ implies that $\psi_k(a)\smajtau \psi_k(b)$ in $N_k$. By Lemma \ref{Nkhasstrict} the C*-algebra algebra $N_k$ has the property of strict comparison of positive elements. Hence,
by  Theorem \ref{uniformsub},    there exists a number $N\in \N$ and elements $d_{k,1},\dots,d_{k,N}\in N_k$
	such that 
	\[
	\Big\|\psi_k(a)-\sum_{i=1}^N d_{k,i}^*\psi_k(b)d_{k,i}\Big\|<\frac{\epsilon}{m+1}\hbox{ and }  \sum_{i=1}^N d_{k,i}d_{k,i}^*\leq 1.
	\]  
	The number $N$ depends only on $\epsilon$ and $m$. If we set $y_{k,i}=(\psi_k(b))^{\frac 1 2}d_{k,i}$ then we can  rewrite these inequalities  as 
	\[
\Big\|\psi_k(a)-\sum_{i=1}^N y_{k,i}^*y_{k,i}\Big\|<\frac{\epsilon}{m+1}\hbox{ and }  \sum_{i=1}^N y_{k,i}y_{k,i}^*\leq \psi_k(b),
\] 
 Applying 
	$\phi_k$ on both inequalities and using that it is an order zero map  we deduce that 
	\[
	\Big\|\phi_k\psi_k(a)-\sum_{i=1}^N \tilde y_{k,i}^*\tilde y_{k,i}\Big\|<\frac{\epsilon}{m+1}\hbox{ and 
}\sum_{i=1}^N \tilde y_{k,i}\tilde y_{k,i}^*\leq \phi_k\psi_k(b)
\] for some $\tilde y_{k,i}\in A_\infty$. Adding over all $k$ we get that
\[
\Big\|a-\sum_{k=0}^m\sum_{i=1}^N \tilde y_{k,i}^*\tilde y_{k,i}\Big\|<\epsilon\hbox{ and }  \sum_{k=0}^m\sum_{i=1}^N \tilde y_{k,i}\tilde y_{k,i}^*\leq b.
\]
We can lift the elements $\tilde y_{k,i}$ to $\prod_\lambda A$ so that these inequalities are preserved. Then from those lifts we find  elements $x_{k,i}\in A$ such that the same inequalities hold in $A$; namely, 
\[
\Big\|a-\sum_{k=0}^m\sum_{i=1}^N x_{k,i}^*x_{k,i}\Big\|<\epsilon\hbox{ and }  \sum_{k=0}^m\sum_{i=1}^N x_{k,i}x_{k,i}^*\leq b.
\] By a well-known lemma of Kirchberg and R{\o}rdam, if $\|a-a'\|<\epsilon$ then $(a-\epsilon)_+=d a' d^*$ for some contraction $d\in A$ (\cite[Lemma 2.2]{kirchberg-rordam0}). Applying this lemma with 
$a'=\sum_{k=0}^m\sum_{i=1}^N x_{k,i}^*x_{k,i}$ we  can turn the inequalities above into the relations claimed by the lemma.
\end{proof}

\begin{proposition}\label{prodnuc}
	Let $A_1,A_2,\ldots$ be a sequence of C*-algebras with uniformly bounded nuclear dimensions. Let 
	$a=(a_n)_{n=1}^\infty$ and $b=(b_n)_{n=1}^\infty$ be positive elements in 
	$\prod_{n=1}^\infty A_n$
	such that  $a_n\smajtau b_n$ for all $n$. Then $a\smajtau b$ in $\prod_{n=1}^\infty A_n$
\end{proposition}
\begin{proof}
	It suffices to show that $\tau(a)\leq \tau(b)$ for all $\tau\in \T(\prod_{n=1}^\infty A_n)$, for then the same argument applied to $(a-t)_+$ and $(b-t)_+$ in place of $a$ and $b$
	gives us  that $\tau((a-t)_+)\leq \tau((b-t)_+)$ for all $\tau$. Let $\epsilon>0$. From the previous lemma we deduce that for each $n$ there exist $x_{1,n},\ldots, x_{N(m+1),n}\in A_n$ such that 
	\[
	(a_n-\epsilon)_+=\sum_{i=1}^{N(m+1)} x_{i,n}^*x_{i,n} \hbox{ and }
	\sum_{i=1}^{N(m+1)} x_{i,n}x_{i,n}^*\leq b_n.
	\]
	The sequences $(x_{i,n})_n$ are necessarily bounded.
	So if we set $x_i=(x_{i,n})_n\in \prod_{n=1}^\infty A_n$ then
	\[
	(a-\epsilon)_+=\sum_{i=1}^{N(m+1)} x_{i}^*x_{i} \hbox{ and }
	\sum_{i=1}^{N(m+1)} x_{i}x_{i}^*\leq b.
	\]
	This implies that $\tau((a-\epsilon)_+)\leq \tau(b)$ for all lower semicontinuous traces $\tau$ on $\prod_{n=1}^\infty A_n$. Since $\epsilon>0$ is arbitrary, we get that $\tau(a)\leq \tau(b)$ for all $\tau$, as desired.
\end{proof}

\begin{theorem}\label{uniformnuc}
	Let $m\in \N$.
	For every $\epsilon>0$ there exists $N$ such that if $A$ is a unital C*-algebra with nuclear dimension at most $m$  and $a,b\in A$ are selfadjoint contractions such that 
	$
	a\in \overline{\co\{ubu^*\mid u\in \U(A)\}}
	$  
	then 
	\[
	\Big\|a-\frac 1 N\sum_{i=1}^N u_iau_i^*\Big\|<\epsilon
	\]
	for some $u_1,\dots,u_N\in \U(A)$.
\end{theorem}
\begin{proof}
	The same proof of Theorem \ref{uniformthm} applies here relying on Proposition \ref{prodnuc} rather than on Proposition \ref{prodstrict}.
\end{proof}

\begin{example}\label{nouniformmaj}
In \cite[Theorem 1.4]{robert} an example is given of a simple unital C*-algebra $A$ with a unique tracial state $\tau$ such that for each $n\in \N$ there exists a selfadjoint element $a_n\in A$ of norm $1$ such that  $\tau(a_n)=0$
and the distance from $a_n$ to the set $\{\sum_{i=1}^n [b_i^*,b_i]\mid b_i\in A\}$ is 1. In this C*-algebra the property of uniform majorization cannot hold.
Indeed, by Haagerup and Zsido's theorem from \cite{haag-zsido}, we have $0\maj a_n$  for all $n$. We claim, however, that no convex combination of at most $n$ unitary conjugates
of $a_n$ can have norm  less than 1. For suppose that there were  unitaries $u_1,\dots,u_n\in A$ such that
\[
\Big\|\sum_{i=1}^n t_iu_i a_nu_i^*\Big\|<1,
\] 
for some $t_i\in [0,1]$ such that $\sum_{i=1}^n t_i=1$. Then 
\[
\Big\|a_n-\sum _{i=1}^n [b_i^*,b_i]\Big\|<1,
\]
where $b_i=t_i^{\frac1  2}u_i(1+a_n)^{\frac 1 2}$ for  $i=1,\dots,n$. This contradicts the property of $a_n$. Thus, no such unitaries  exist.
\end{example}

\begin{theorem}
Let $A_1,A_2,\dots$ be  unital C*-algebras with strict comparison of positive elements by traces or with a uniform bound on their nuclear dimensions. Let 
$A=\prod_{i=1}^\infty A_i/\bigoplus_{i=1}^\infty A_i$ and let $B\subseteq A$ be a separable C*-subalgebra.
	Then for each $a\in A_{\sa}$ we have that
	\[
	\overline{\co(\{uau^*\mid u\in \U(A)\})}\cap (B'\cap A)\neq \varnothing.
	\]
\end{theorem}	
\begin{proof}
Let $(a_n)_n\in \prod_{n=1}^\infty A_n$ be a lift of $a$ with $a_n\in (A_n)_{\sa}$  and $\|a_n\|\leq \|a\|$ for all $n$. 
Let $(b_n^{(1)})_n,(b_n^{(2)})_n,\ldots\in \prod_{n=1}^\infty A_n$
be lifts of a sequence $b^{(1)},b^{(2)},\ldots\in B$ dense in $B$. 
\cite[Lemma 6.4]{kirchberg-rordam} asserts that given an element and a finite set in a C*-algebra we can find a convex combination of unitary conjugates of the given element that almost commutes with the given finite set. (This is derived from Dixmier's approximation property in $A^{**}$.)  Applying this lemma, we can find  for each $a_n\in A_n$ a selfadjoint element $a_n'\maj a_n$
such that $\|[a_n',b_i^{(j)}]\|\leq \frac 1 n\|a\|\|b^{(j)}\|$  for all $1\leq i,j\leq n$. 
Let $a'$ denote the image of $(a_n')_n$ in $A$. Then $a'$
commutes with $b^{(j)}$ for all $j$, and so $a'\in B'\cap A$.
On the other hand, from the fact that $a_n'\maj a_n$ for all $n$ we get that
$(a_n')_n\maj (a_n)_n$ in $\prod_{n=1}^\infty A_n$. In the case that all the C*-algebras have strict comparison by traces, this  follows from Proposition  \ref{prodstrict}. If their
nuclear dimensions  are uniformly bounded, this follows from Proposition \ref{prodnuc}.  Passing to the quotient we get that  $a'\maj a$ in $A$.
That is, $a'\in \overline{\co(\{uau^*\mid u\in \U(A)\})}$.
\end{proof}

\end{document}